\documentclass[final,onefignum,onetabnum,hidelinks]{siamart251216}

\usepackage{amsfonts}
\usepackage{amsmath}
\usepackage{amssymb}
\usepackage{graphicx}
\usepackage{epstopdf}
\usepackage{subfigure}
\usepackage{booktabs}
\usepackage{multirow}
\usepackage{tabularx}
\usepackage{makecell}
\usepackage{algorithm}
\usepackage{algpseudocode}
\usepackage{tikz}

\ifpdf
  \DeclareGraphicsExtensions{.eps,.pdf,.png,.jpg}
\else
  \DeclareGraphicsExtensions{.eps}
\fi

\renewcommand{\vec}[1]{\boldsymbol{#1}}

\newsiamremark{remark}{Remark}
\newsiamremark{example}{Example}
\newsiamremark{hypothesis}{Hypothesis}
\crefname{hypothesis}{Hypothesis}{Hypotheses}
\newsiamthm{claim}{Claim}

\newcommand{\va}{{\vec a}}
\newcommand{\rr}{{\bf r}}

\newcommand{\bX}{\bar{\mathcal{X}}}

\usepackage{amsopn}

\headers{TT-Format LRMC via Riemannian Optimization}{Z. Hu, Z. Chen, D. Sun, and L. Zhang}

\title{Low-Rank Tensor Completion using Tensor Train Decomposition via Riemannian Optimization on the Quotient Geometry \thanks{Submitted to the editors DATE.
\funding{The work of Liping Zhang is supported by the National Natural Science Foundation of China (Grant No. 12571323). Defeng Sun's research is supported by the Hong Kong RGC Senior Research Fellow Scheme [No. SRFS22235S02] and the GRF Grant 15307822.}}}

\author{Zhenlong Hu\thanks{Department of Mathematical Sciences, Tsinghua University, Beijing 100084, China.}
\and Zhongming Chen\thanks{Department of Mathematics, Hangzhou Dianzi University, Hangzhou 310018, China.}
\and Defeng Sun\thanks{Department of Applied Mathematics, The Hong Kong Polytechnic University, Hung Hom, Kowloon, Hong Kong (\email{defeng.sun@polyu.edu.hk}). }
\and Liping Zhang\thanks{Corresponding author. Department of Mathematical Sciences, Tsinghua University, Beijing 100084, China (\email{lipingzhang@tsinghua.edu.cn}).}
}

\begin{document}

\maketitle

\begin{abstract}
Owing to the effectiveness of Tensor Train (TT) decomposition in managing high-order tensors, low-rank tensor completion within the TT-format has emerged as a prominent research focus. In this paper, we leverage the left-orthogonal property of the TT-decomposition to construct a novel quotient manifold and introduce a family of admissible Riemannian metrics. Within this geometric framework, we propose a new approach to constructing retractions compatible with the quotient structure, realized via two novel retractions based on recursive polar and QR decompositions that respect the recursive orthogonalization structure of the TT format. We then derive Riemannian gradient descent and conjugate gradient methods to solve the tensor completion problem. Theoretically, our approach streamlines the horizontal projection by reducing the number of unknowns per block from a quadratic dependence on the TT-ranks to a near-half scaling, thereby enhancing computational efficiency over conventional quotient-based methods. Numerical experiments demonstrate that the proposed algorithms achieve reconstruction accuracy comparable to state-of-the-art TT-based geometric methods.
\end{abstract}

\begin{keywords}
tensor completion, tensor train decomposition, Riemannian optimization, Riemannian metric, retraction, quotient manifold
\end{keywords}

\begin{MSCcodes}
90C30, 65K05, 15A69, 65F99
\end{MSCcodes}

\section{Introduction}\label{sec:introduction}
Tensors serve as multidimensional generalizations of matrices, providing a powerful framework for modeling high-dimensional data in fields such as data mining \cite{acar2005modeling} and computer vision \cite{vasilescu2002multilinear}. A fundamental computational challenge in this context is tensor completion, which seeks to reconstruct the missing entries of a partially observed tensor. To render this ill-posed problem tractable, low-rank models based on various tensor decompositions are widely employed. Among these, the Tensor Train (TT) decomposition \cite{oseledets2011tensor} offers an exceptionally parameter-efficient representation for high-order tensors. By ensuring that storage costs scale only linearly with the tensor order, the TT-format effectively circumvents the curse of dimensionality. This paper addresses the tensor completion problem under the assumption of a low TT-rank, formulated as the following optimization problem:
\begin{equation}\label{eqttrank}
\min_{\mathrm{rank_{TT}}{\mathcal{X}} =\rr} \; F(\mathcal{X}) = \frac{1}{2} \bigl\| \mathcal{P}_\Omega\bigl(\mathcal{X}\bigr) - \mathcal{P}_\Omega(\Gamma) \bigr\|_F^2,
\end{equation}
where $\Gamma \in \mathbb{R}^{n_1 \times \cdots \times n_d}$
is the partially observed target tensor, $\Omega$ denotes the index set of observed entries, and the sampling operator $\mathcal{P}_\Omega$ is defined as:
\[
\mathcal{P}_\Omega(\mathcal{X})(i_1,\dots,i_d) =
\begin{cases}
\mathcal{X}(i_1,\dots,i_d), & \text{if } (i_1,\dots,i_d) \in \Omega, \\[2pt]
0, & \text{otherwise}.
\end{cases}
\]
For second-order tensors, the problem simplifies to the well-studied matrix completion task \cite{candes2012exact,keshavan2009matrix}. Consequently, extending completion techniques from matrices to higher-order tensors represents a natural research trajectory \cite{liu2012tensor,zhang2016exact}. Recently, manifold optimization has emerged as a robust alternative framework for constrained optimization. Compared to conventional methods, manifold-based algorithms often exhibit superior numerical properties \cite{absil2008optimization}, and their application to tensor completion has garnered significant attention.

Existing literature on manifold-based tensor completion primarily employs two distinct geometric strategies. One approach is to embed the set of fixed-rank tensors directly into a Euclidean space, an intuition applied to matrices \cite{vandereycken2013low}, Tucker tensors \cite{kressner2014low}, and TT tensors \cite{holtz2012manifolds,steinlechner2016riemannian}. To handle the gauge symmetries inherent in tensor decompositions, a second strategy models the factors as a product manifold and identifies equivalent classes via quotienting. This framework has gained traction for various formats, including matrices \cite{tong2021accelerating,wei2016guarantees}, Tucker \cite{kasai2016low}, and TT tensors \cite{cai2022tensor}. There also exist methods that operate directly on the product manifold without quotienting, employing a preconditioned metric \cite{gao2024riemannian,gao2025optimization}. 

The choice of Riemannian metric critically affects the convergence behavior of optimization algorithms \cite{luo2024geometric,mishra2016riemannian,zheng2025riemannian}. In quotient geometry, however, the requirement that the action of the quotient group be isometric imposes a strong restriction on admissible metrics: the metric on the total space must be preserved under the group action \cite{absil2008optimization,boumal2023introduction}. When this condition holds, the quotient map becomes an isometric submersion, which induces a well-defined Riemannian metric on the quotient manifold \cite{absil2014two,dong2022analysis,luo2024geometric}. In all existing quotient-based works on tensor completion, including those for matrices \cite{mishra2012riemannian,mishra2014r3mc}, Tucker tensors \cite{kasai2016low}, and TT tensors \cite{cai2022tensor}, this is achieved by adopting a single, complex Riemannian metric that couples all factor cores, which we refer to as {\bf full metric}. As a result, operations such as the Riemannian gradient or the tangent projection become expensive because they inevitably involve all cores.

Yet, the geometric structure of a quotient manifold is intrinsically linked to the choice of the equivalence group; different choices can lead to distinct invariance conditions and thus affect the flexibility in selecting Riemannian metrics. However, regardless of whether the quotient group is the general linear group $\mathrm{GL}(r_k)$ (as in \cite{cai2022tensor,mishra2012riemannian}) or the orthogonal group $\mathrm{O}(r_k)$ (as in \cite{kasai2016low,mishra2014r3mc}), all existing works have relied on the full metric. The present work is the first to systematically construct a family of admissible metrics on an orthogonal quotient manifold.

Beyond the choice of metric, retractions on a quotient manifold must also be compatible with the equivalence relation to descend to the quotient, which is a property that standard product-manifold constructions cannot always guarantee. Our approach overcomes this challenge through a key insight: if a retraction preserves the tensor formed by the direct sum of the base point and the tangent vector, the quotient compatibility condition is automatically satisfied, which converts the problem into the simpler task of constructing such a retraction. We realize this by adapting the recursive left-orthogonalization procedure of the TT format.

In this work, we consider the left-orthogonal TT format, where all but the last TT core satisfy an orthogonality condition; specifically, the left-unfolding matrices $L(\mathcal{X}_k)$ are orthogonal for $k=1,\dots,d-1$. This is a well-established representation widely used in practice \cite{cai2022provable,holtz2012manifolds,lubich2015time,steinlechner2016riemannian}. Building on this format, we construct a quotient manifold by identifying left-orthogonal TT cores under the action of the orthogonal group. In this framework, the horizontal projection requires solving a linear system derived from coupled Lyapunov equations, and we show that the resulting coefficient matrix is block-tridiagonal. This formulation offers a more efficient geometric framework, particularly for tensors with large ranks. Crucially, it allows us to construct a family of $d!$ admissible Riemannian metrics that satisfy the isometric condition without being restricted to the full metric. To our knowledge, this is the first systematic construction of such a metric family on an orthogonal quotient manifold for TT tensors. Leveraging the proposed metrics and retractions, we implement first-order Riemannian optimization algorithms, including gradient descent and conjugate gradient methods. Numerical experiments validate our approach, demonstrating performance comparable to state-of-the-art TT-based manifold optimization methods while highlighting the practical advantages of selecting different metrics from our proposed family.

The main contributions are summarized as follows:
\begin{itemize}
    \item We construct a novel quotient manifold for the left-orthogonal TT format by employing an orthogonal group action. This formulation reduces the horizontal projection to a block-tridiagonal linear system derived from coupled Lyapunov equations. By exploiting skew-symmetry, we reduce the system size from  $\sum_{k=1}^{d-1} r_k^2$ \cite{cai2022tensor} to $\sum_{k=1}^{d-1} r_k(r_k-1)/2$. This significantly lowers the computational complexity of characterizing the manifold's tangent space, particularly for high-rank tensors, while maintaining a rigorous geometric structure.

  \item We introduce the first family of $d!$ admissible Riemannian metrics for TT quotient manifolds. By combining the left-orthogonal structure with an orthogonal group action, these metrics naturally satisfy the isometry condition. This framework offers a tunable trade-off between computational cost and geometric fidelity, ranging from the inexpensive Euclidean metric to the information-rich full metric. This flexibility allows for optimizing convergence behavior relative to available computational resources, unexplored in previous tensor optimization literature.

\item We introduce a new approach to constructing retractions compatible with the quotient structure on the TT tensor manifold, realized via two novel retractions based on recursive polar and QR decompositions that respect the recursive orthogonalization structure of the TT format.
Combined with the proposed family of metrics, these tools enable the implementation of first-order Riemannian algorithms, including RGD and RCG, on the quotient manifold.
Numerical experiments demonstrate that our approach achieves reconstruction accuracy comparable to state-of-the-art embedded and quotient methods, while highlighting the practical advantages of selecting different metrics from the proposed family to balance computational efficiency and convergence speed.

\end{itemize}

The remainder of the paper is organized as follows. Section~\ref{sec:preliminaries} presents necessary preliminaries. Section~\ref{sec:quotient} details the construction of the proposed quotient manifold. Section~\ref{sec:algorithms} applies Riemannian gradient descent and conjugate gradient methods on this manifold. Section~\ref{sec:experiments} presents numerical experiments, including an analysis of algorithmic components and a comparison with existing methods. Section~\ref{sec:conclusion} concludes the paper.

\section{Notation and Preliminaries}
\label{sec:preliminaries}

\subsection{Notation}
Throughout this paper, scalars are denoted by lowercase letters (e.g., $x$), vectors by bold lowercase letters (e.g., $\mathbf{x}$), matrices by capital letters (e.g., $X$), and tensors of order three or higher by calligraphic letters (e.g., $\mathcal{X}$).
For a positive integer $n$, we write $[n] = \{1, 2, \dots, n\}$.
For a matrix $A \in \mathbb{R}^{m \times n}$, its $i$th row and $j$th column are denoted by $A(i,:)$ and $A(:,j)$, respectively.
The Kronecker product is denoted by $\otimes$, and the $n \times n$ identity matrix by $I_n$. The zero matrix of size $m \times n$ is denoted by $0_{m \times n}$; for a square matrix, we write $0_n$.
The sets of $n \times n$ symmetric and skew-symmetric matrices are denoted by $\operatorname{Sym}(n)$ and $\operatorname{Skew}(n)$, respectively. We denote by $\mathrm{GL}(r)$ the general linear group of $r \times r$ invertible matrices, and by $\mathrm{O}(r)$ the orthogonal group of $r \times r$ matrices satisfying $Q^\top Q = I_r$.

We adopt a multi-index notation: for indices $i_j \in [n_j]$ ($j \in [d]$), the integer $\overline{i_1 i_2 \cdots i_d}$ in $\prod_{k=1}^d n_k$ is defined as
\[
\overline{i_1 i_2 \cdots i_d} = i_1 + (i_2 - 1)n_1 + \cdots + (i_d - 1)n_1 \cdots n_{d-1}.
\]

For a third-order tensor $\mathcal{A} \in \mathbb{R}^{n_1 \times n_2 \times n_3}$, its $i$th slice is denoted by $\mathcal{A}(i) \in \mathbb{R}^{n_1 \times n_3}$.
The left unfolding $L(\mathcal{A}) \in \mathbb{R}^{n_1 n_2 \times n_3}$ and the right unfolding $R(\mathcal{A}) \in \mathbb{R}^{n_1 \times n_2 n_3}$ are defined by
\[
L(\mathcal{A})(\overline{i_1 i_2}, i_3) = \mathcal{A}(i_1, i_2, i_3), \quad
R(\mathcal{A})(i_1, \overline{i_2 i_3}) = \mathcal{A}(i_1, i_2, i_3).
\]

For a $d$th-order tensor $\mathcal{A} \in \mathbb{R}^{n_1 \times n_2 \times \cdots \times n_d}$, its mode-$k$ matricization $\mathcal{A}_{(k)} \in \mathbb{R}^{n_k \times \prod_{i \neq k} n_i}$ and its $k$th unfolding matrix $\mathcal{A}_{\langle k \rangle} \in \mathbb{R}^{(n_1 \cdots n_k) \times (n_{k+1} \cdots n_d)}$ are given for $k \in [d]$ by
\begin{align*}
\mathcal{A}_{(k)}\bigl(i_k, \overline{i_1 \cdots i_{k-1} i_{k+1} \cdots i_d}\bigr) &= \mathcal{A}(i_1, i_2, \ldots, i_d), \\
\mathcal{A}_{\langle k \rangle}\bigl(\overline{i_1 \cdots i_k}, \overline{i_{k+1} \cdots i_d}\bigr) &= \mathcal{A}(i_1, i_2, \ldots, i_d).
\end{align*}

The Frobenius norm of $\mathcal{A}$ is
\[
\|\mathcal{A}\|_F = \sqrt{\sum_{i_1=1}^{n_1} \sum_{i_2=1}^{n_2} \cdots \sum_{i_d=1}^{n_d} \mathcal{A}(i_1, i_2, \ldots, i_d)^2}.
\]

For a tensor $\mathcal{X} \in \mathbb{R}^{n_1 \times \cdots \times n_d}$ and a matrix $A \in \mathbb{R}^{m \times n_k}$, the $k$-mode product $\mathcal{X} \times_k A$ yields a tensor in $\mathbb{R}^{n_1 \times \cdots \times n_{k-1} \times m \times n_{k+1} \times \cdots \times n_d}$ with entries
\[
(\mathcal{X} \times_k A)(i_1, \ldots, i_{k-1}, j, i_{k+1}, \ldots, i_d) = \sum_{i_k=1}^{n_k} \mathcal{X}(i_1, \ldots, i_{k-1}, i_k, i_{k+1}, \ldots, i_d) \, A(j, i_k).
\]

From these definitions, one obtains the following relations for any $\mathcal{X} \in \mathbb{R}^{n_1 \times n_2 \times n_3}$, $P \in \mathbb{R}^{m_1 \times n_1}$, and $Q \in \mathbb{R}^{m_3 \times n_3}$:
\begin{align}
L(\mathcal{X} \times_1 P \times_3 Q) &= (I_{n_2} \otimes P) \, L(\mathcal{X}) \, Q^{\!\top}, \label{eq:Ltrans} \\
R(\mathcal{X} \times_1 P \times_3 Q) &= P \, R(\mathcal{X}) \, (Q \otimes I_{n_2})^{\!\top}. \label{eq:Rtrans}
\end{align}

\subsection{TT Decomposition}

The TT decomposition represents a higher-order tensor as a product of third-order tensors called TT cores. For any tensor $\mathcal{X} \in \mathbb{R}^{n_1 \times n_2 \times \cdots \times n_d}$, each element is expressed as
\begin{equation}\label{eq_ttd}
\mathcal{X}(i_1, i_2, \ldots, i_d) = \mathcal{X}_1(i_1) \mathcal{X}_2(i_2) \cdots \mathcal{X}_d(i_d),
\end{equation}
where each $\mathcal{X}_k \in \mathbb{R}^{r_{k-1} \times n_k \times r_k}$ for $k \in [d]$ is a TT core. When the context is clear, we also refer to them simply as cores or core tensors. To ensure the matrix product yields a scalar, we set $r_0 = r_d = 1$ \cite{oseledets2011tensor}.

The TT-rank of $\mathcal{X}$ is defined as the smallest tuple $\rr = (r_0, r_1, \ldots, r_d)$ for which a TT decomposition exists with cores of dimensions $r_{k-1} \times n_k \times r_k$, and is denoted by $\operatorname{rank}_{\mathrm{TT}}(\mathcal{X}) = \rr$.
If a tensor admits a decomposition as in \eqref{eq_ttd}, then necessarily $r_k \ge \operatorname{rank}(\mathcal{X}_{\langle k \rangle})$ for $k \in [d-1]$, where $\mathcal{X}_{\langle k \rangle}$ is the $k$th unfolding matrix. Moreover, the SVD-based TT algorithm \cite{oseledets2011tensor} constructs a decomposition with $r_k = \operatorname{rank}(\mathcal{X}_{\langle k \rangle})$ for $k \in [d-1]$. Hence,
\[
\operatorname{rank}_{\mathrm{TT}}(\mathcal{X}) = \bigl(1, \operatorname{rank}(\mathcal{X}_{\langle 1 \rangle}), \ldots, \operatorname{rank}(\mathcal{X}_{\langle d-1 \rangle}), 1\bigr).
\]

The TT-rank also admits an equivalent characterization via the ranks of the core unfoldings \cite{holtz2012manifolds}:

\begin{proposition}\label{tt-rank}
Let $\mathcal{X} \in \mathbb{R}^{n_1 \times n_2 \times \cdots \times n_d}$ have a TT decomposition \eqref{eq_ttd} with cores $\mathcal{X}_k \in \mathbb{R}^{r_{k-1} \times n_k \times r_k}$ for $k \in [d]$. Then $\operatorname{rank}_{\mathrm{TT}}(\mathcal{X}) = (r_0, r_1, \ldots, r_d)$ if and only if $\operatorname{rank}\bigl(L(\mathcal{X}_k)\bigr) = r_k$ and $\operatorname{rank}\bigl(R(\mathcal{X}_k)\bigr) = r_{k-1}$ for every $k \in [d]$.
\end{proposition}

The tensor $\mathcal{X}$ defined by \eqref{eq_ttd} is equivalently written as
\begin{equation}\label{tt-map}
\mathcal{X} = \Phi(\mathcal{X}_1, \mathcal{X}_2, \dots, \mathcal{X}_d),
\end{equation}
where $\Phi$ maps the sequence of TT cores to the full tensor via the product in \eqref{eq_ttd}.

Setting $n = \max_{1\le k\le d} n_k$ and $r = \max_{0\le k\le d} r_k$, the storage cost of $\mathcal{X}$ in TT format is $\mathcal{O}(d n r^2)$, which avoids exponential scaling with the order $d$. Thus, the TT decomposition provides an efficient representation for high-order data.

We review the interface matrices \cite{steinlechner2016riemannian,holtz2012manifolds} and then introduce a generalized definition that is central to our work.
For each $k \in [d]$, the interface matrices $\mathcal{X}_{\le k} \in \mathbb{R}^{(n_1 \cdots n_k) \times r_k}$ and $\mathcal{X}_{\ge k} \in \mathbb{R}^{(n_k \cdots n_d) \times r_{k-1}}$ are defined by
\begin{align*}
\mathcal{X}_{\le k} &= \operatorname{reshape}\!\Bigl(\Phi(\mathcal{X}_1, \dots, \mathcal{X}_k),\; \prod_{i=1}^k n_i,\; r_k \Bigr), \\[2pt]
\mathcal{X}_{\ge k} &= \Bigl( \operatorname{reshape}\!\Bigl(\Phi(\mathcal{X}_k, \dots, \mathcal{X}_d),\; r_{k-1},\; \prod_{i=k}^d n_i \Bigr) \Bigr)^{\!\top},
\end{align*}
where the $\operatorname{reshape}$ operator rearranges the entries into a matrix of the prescribed shape. These matrices satisfy the recursions
\begin{equation}\label{eq_inte}
\mathcal{X}_{\le k} = (I_{n_k} \otimes \mathcal{X}_{\le k-1}) \, L(\mathcal{X}_k), \quad
\mathcal{X}_{\ge k} = (\mathcal{X}_{\ge k+1} \otimes I_{n_k}) \, R(\mathcal{X}_k)^{\!\top},
\end{equation}
with the conventions $\mathcal{X}_{\le 0} = \mathcal{X}_{\ge d+1} = 1$.

We now generalize these constructs. For $k \in [d]$ and $u_k \in [k]$, define $\mathcal{X}_{\le k, u_k} \in \mathbb{R}^{r_{k-u_k} n_{k-u_k} \cdots n_k \times r_k}$ by
\[
\mathcal{X}_{\le k, u_k} = \operatorname{reshape}\!\Bigl( \Phi(\mathcal{X}_{k-u_k+1}, \dots, \mathcal{X}_k),\; r_{k-u_k} \prod_{i=k-u_k+1}^{k} n_i,\; r_k \Bigr).
\]
Similarly, for $k \in [d]$ and $a_k \in [d-k+1]$, define $\mathcal{X}_{\ge k, a_k} \in \mathbb{R}^{n_k \cdots n_{k+a_k-1} r_{k+a_k-1} \times r_{k-1}}$ by
\[
\mathcal{X}_{\ge k, a_k} = \Bigl( \operatorname{reshape}\!\Bigl( \Phi(\mathcal{X}_k, \dots, \mathcal{X}_{k+a_k-1}),\; r_{k-1},\; r_{k+a_k-1} \prod_{i=k}^{k+a_k-1} n_i \Bigr) \Bigr)^{\!\top}.
\]
Here $u_k$ (respectively $a_k$) counts how many consecutive TT cores, starting from the $k$th core, are contracted to the left (respectively right). For the boundary cases we set
\[
\mathcal{X}_{\le k, 0} = I_{r_k}, \qquad \mathcal{X}_{\ge k, 0} = I_{r_{k-1}}.
\]

Define the matrices
\begin{equation*}
L^{k, u_k} = \mathcal{X}_{\le k, u_k}^{\!\top} \mathcal{X}_{\le k, u_k}, \quad
R^{k, a_k} = \mathcal{X}_{\ge k, a_k}^{\!\top} \mathcal{X}_{\ge k, a_k},
\end{equation*}
for $k \in [d]$, $u_k \in \{0,\dots,k\}$, and $a_k \in \{0,\dots,d-k+1\}$, with the convention $L^{0, u_0} = R^{d+1, a_{d+1}} = 1$.
When the indices $u_k$ and $a_k$ are clear from context, we write $L^{k}$ and $R^{k}$.
The extended interface matrices satisfy the recursive relations
\begin{equation}\label{eq_fur_inte}
\begin{aligned}
\mathcal{X}_{\le k} &= (I_{n_{k-u_k+1} \times \cdots \times n_k} \otimes \mathcal{X}_{\le k-u_k}) \, \mathcal{X}_{\le k, u_k}, \\
\mathcal{X}_{\ge k} &= (\mathcal{X}_{\ge k+a_k} \otimes I_{n_k \times \cdots \times n_{k+a_k-1}}) \, \mathcal{X}_{\ge k, a_k}.
\end{aligned}
\end{equation}
From the definitions, the $k$th unfolding of $\mathcal{X}$ can be written as
\begin{equation}\label{eq_mk}
\mathcal{X}_{\langle k \rangle} = \mathcal{X}_{\le k} \, \mathcal{X}_{\ge k+1}^{\!\top}.
\end{equation}

A set of TT cores $\{\mathcal{X}_k \in \mathbb{R}^{r_{k-1} \times n_k \times r_k} : k \in [d] \}$ is called $i$-orthogonal \cite{steinlechner2016riemannian} if
\begin{equation}\label{eq_lo}
\begin{aligned}
L(\mathcal{X}_k)^{\!\top} L(\mathcal{X}_k) &= I_{r_k}, & k &= 1, \dots, i-1, \\
R(\mathcal{X}_k) R(\mathcal{X}_k)^{\!\top} &= I_{r_{k-1}}, & k &= i+1, \dots, d.
\end{aligned}
\end{equation}
When $i = d$, the cores are said to be left-orthogonal. This left-orthogonal TT format is a commonly used representation in practice \cite{steinlechner2016riemannian, holtz2012manifolds, cai2022provable}.
\begin{proposition}\label{tt-uniq}
Let $\mathcal{X} \in \mathbb{R}^{n_1 \times n_2 \times \cdots \times n_d}$ satisfy $\operatorname{rank}_{\mathrm{TT}}(\mathcal{X}) = (r_0, r_1, \ldots, r_d)$.
Then there exists a TT decomposition \eqref{eq_ttd} of $\mathcal{X}$ with left-orthogonal cores, i.e., $L(\mathcal{X}_k)^{\!\top} L(\mathcal{X}_k) = I_{r_k}$ for all $k \in [d-1]$.
Furthermore, if two such left-orthogonal decompositions satisfy
\[
\Phi(\mathcal{X}_1, \dots, \mathcal{X}_d) = \Phi(\mathcal{Y}_1, \dots, \mathcal{Y}_d),
\]
then there exist orthogonal matrices $Q_k \in \mathbb{R}^{r_k \times r_k}$ ($k \in [d-1]$) such that
\[
\mathcal{Y}_k = \mathcal{X}_k \times_1 Q_{k-1} \times_3 Q_k \qquad (k \in [d]),
\]
with $Q_0 = Q_d = 1$.
\end{proposition}

\subsection{Geometric Structure of the Manifold}
Let $\mathcal{N}_{\rr}$ denote the set of tensors with fixed TT-rank $\rr = (r_0, r_1, \dots, r_d)$; that is,
\[
\mathcal{N}_{\rr} = \{\mathcal{X} \in \mathbb{R}^{n_1 \times n_2 \times \cdots \times n_d} : \operatorname{rank}_{\mathrm{TT}}(\mathcal{X}) = \rr\},
\]
where $r_0 = r_d = 1$. The necessary and sufficient conditions for $\mathcal{N}_{\rr}$ to be nonempty are \cite{uschmajew2020geometric}
\begin{equation}\label{eqn_cond}
r_{k-1} \le n_k r_k, \qquad r_k \le n_k r_{k-1}, \qquad k \in [d].
\end{equation}
Throughout this paper we assume that the TT-rank vector $\rr$ satisfies \eqref{eqn_cond}. The set $\mathcal{N}_{\rr}$ is a smooth manifold of dimension
$\sum_{k=1}^d (r_{k-1} n_k r_k) - \sum_{k=1}^{d-1} r_k^2$ \cite{holtz2012manifolds}.

Let $\mathbb{R}^{r_{k-1}\times n_k\times r_k}_*$ denote the set of third-order tensors whose two unfolding matrices are both of full rank; i.e.,
\[
\mathbb{R}^{r_{k-1}\times n_k\times r_k}_* = \bigl\{\mathcal{A}\in\mathbb{R}^{r_{k-1}\times n_k\times r_k}:
\operatorname{rank}(L(\mathcal{A})) = r_k,\ \operatorname{rank}(R(\mathcal{A})) = r_{k-1}\bigr\}.
\]
A parameterization of $\mathcal{N}_{\rr}$ is given by
\[
\begin{aligned}
\overline{\mathcal{M}}_{\rr} = \bigl\{ &(\mathcal{X}_1,\dots,\mathcal{X}_d)\in \mathbb{R}^{r_{0}\times n_1\times r_1}_* \times \cdots \times \mathbb{R}^{r_{d-1}\times n_d\times r_d}_* : \\
& L(\mathcal{X}_k)^{\!\top} L(\mathcal{X}_k) = I_{r_k},\ k \in [d-1] \bigr\}.
\end{aligned}\]

By Proposition~\ref{tt-rank}, the mapping $\Phi$ defined in \eqref{tt-map} is surjective from $\overline{\mathcal{M}}_{\rr}$ onto $\mathcal{N}_{\rr}$, but it is not injective. For any $\mathcal{X}\in\mathcal{N}_{\rr}$, Proposition~\ref{tt-uniq} guarantees the existence of $\bar{\mathcal{X}}=(\mathcal{X}_1,\dots,\mathcal{X}_d)\in\overline{\mathcal{M}}_{\rr}$ such that $\Phi(\bar{\mathcal{X}})=\mathcal{X}$. If we define transformed cores
\[
\mathcal{Y}_k = \mathcal{X}_k \times_1 Q_{k-1} \times_3 Q_k, \qquad k \in [d],
\]
where $Q_0=Q_d=1$ and $Q_1,\dots,Q_{d-1}$ are orthogonal matrices of appropriate sizes, then $\bar{\mathcal{Y}}=(\mathcal{Y}_1,\dots,\mathcal{Y}_d)$ also belongs to $\overline{\mathcal{M}}_{\rr}$ and satisfies $\Phi(\bar{\mathcal{Y}})=\mathcal{X}$. Proposition~\ref{tt-uniq} further shows that this transformation by orthogonal matrices is the only source of non-uniqueness when $\operatorname{rank}_{\mathrm{TT}}(\mathcal{X})=\rr$. This observation motivates the identification of equivalent elements in $\overline{\mathcal{M}}_{\rr}$ that represent the same tensor, leading naturally to a quotient-manifold description \cite{absil2008optimization}.

Consequently, we endow $\overline{\mathcal{M}}_{\rr}$ with the equivalence relation $\sim$ defined by $\bar{\mathcal{X}} \sim \bar{\mathcal{Y}}$ if and only if $\Phi(\bar{\mathcal{X}}) = \Phi(\bar{\mathcal{Y}})$ for $\bar{\mathcal{X}}=(\mathcal{X}_1,\dots,\mathcal{X}_d),\ \bar{\mathcal{Y}}=(\mathcal{Y}_1,\dots,\mathcal{Y}_d) \in \overline{\mathcal{M}}_{\rr}$. According to Proposition~\ref{tt-uniq}, $\Phi(\bar{\mathcal{X}}) = \Phi(\bar{\mathcal{Y}})$ holds precisely when there exist $Q_k \in \mathrm{O}(r_k)$ ($k\in[d-1]$) such that
\begin{equation}\label{orth-group}
\bar{\mathcal{Y}} = \varphi_Q(\bar{\mathcal{X}}) :=
\bigl(\mathcal{X}_1 \times_3 Q_1,\; \mathcal{X}_2 \times_1 Q_1 \times_3 Q_2,\; \dots,\; \mathcal{X}_d \times_1 Q_{d-1}\bigr).
\end{equation}
The equivalence class containing $\bar{\mathcal{X}}$ is
$
[\bar{\mathcal{X}}] = \{ \bar{\mathcal{Y}} \in \overline{\mathcal{M}}_{\rr} : \bar{\mathcal{Y}} \sim \bar{\mathcal{X}} \}.
$
The quotient manifold of $\overline{\mathcal{M}}_{\rr}$ by $\sim$ is denoted as
\[
\mathcal{M}_{\rr} = \overline{\mathcal{M}}_{\rr}/\!\sim \; = \{ [\bar{\mathcal{X}}] : \bar{\mathcal{X}} \in \overline{\mathcal{M}}_{\rr} \},
\]
and the canonical projection $\pi:\overline{\mathcal{M}}_{\rr} \to \mathcal{M}_{\rr}$ is defined by $\pi(\bar{\mathcal{X}}) = [\bar{\mathcal{X}}]$.

Recently, a related construction was given in \cite{cai2022tensor}, which considers a larger total space without orthogonality constraints:
\[
\overline{\mathcal{F}}_{\rr}=\bigl\{(\mathcal{X}_1,\dots,\mathcal{X}_d)\in \mathbb{R}^{r_{0}\times n_1\times r_1}_* \times \cdots \times \mathbb{R}^{r_{d-1}\times n_d\times r_d}_*  \bigr\}.
\]
Equivalence in $\overline{\mathcal{F}}_{\rr}$ is defined via the action of the general linear group:
\begin{equation}\label{eq:GL-action}
\bar{\mathcal{Y}} \sim' \bar{\mathcal{X}} \;\Longleftrightarrow\; \bar{\mathcal{Y}} = \theta_G(\bar{\mathcal{X}}) := \bigl(\mathcal{X}_1 \times_3 G_1^{\!\top},\; \mathcal{X}_2 \times_1 G_1^{-1} \times_3 G_2^{\!\top},\; \dots,\; \mathcal{X}_d \times_1 G_{d-1}^{-1}\bigr),
\end{equation}
where $G_k \in \mathrm{GL}(r_k)$ for $k\in [d-1]$. The corresponding quotient manifold is $\mathcal{F}_{\rr} = \overline{\mathcal{F}}_{\rr}/\!\sim'$.

Clearly, the total space $\overline{\mathcal{M}}_{\rr}$ is a submanifold of $\overline{\mathcal{F}}_{\rr}$, and the orthogonal gauge transformations \eqref{orth-group} are a restriction of the general linear action \eqref{eq:GL-action} to $Q_k \in \mathrm{O}(r_k) \subset \mathrm{GL}(r_k)$. Imposing the left-orthogonality constraints thus reduces the quotienting group from $\mathrm{GL}(r_k)$ to $\mathrm{O}(r_k)$. As will be shown in Section~\ref{rieme}, this reduction, together with the left-orthogonal structure, enables a family of admissible Riemannian metrics on the quotient manifold $\mathcal{M}_{\rr}$, in contrast to the single metric used in previous quotient-based works.

\section{Quotient Manifold Structure}
\label{sec:quotient}

In this section, we lift the tensor completion problem to the total space $\overline{\mathcal{M}}_{\rr}$ and introduce a family of Riemannian metrics compatible with the quotient structure.

Recall the original problem in the tensor space $\mathcal{N}_{\rr}$:
\[
\min_{\mathcal{X} \in \mathcal{N}_{\rr}} F(\mathcal{X}) = \frac12 \| \mathcal{P}_\Omega(\mathcal{X}) - \mathcal{P}_\Omega(\Gamma) \|_F^2,
\]
where $\mathcal{P}_\Omega$ and $\Gamma$ are defined in \eqref{eqttrank}. Using the parameterization $\mathcal{X} = \Phi(\mathcal{X}_1,\dots,\mathcal{X}_d)$, we lift the problem to the total space:
\begin{equation}\label{problem-manifold}
\min_{\bar{\mathcal{X}} \in \overline{\mathcal{M}}_{\rr}} \; \overline{h}(\bar{\mathcal{X}}) = \frac12 \| \mathcal{P}_\Omega(\Phi(\mathcal{X}_1,\dots,\mathcal{X}_d)) - \mathcal{P}_\Omega(\Gamma) \|_F^2.
\end{equation}
Since $\overline{h}$ depends only on $\Phi(\bar{\mathcal{X}})$, it is invariant under the equivalence relation: $\bar{\mathcal{X}} \sim \bar{\mathcal{Y}}$ implies $\overline{h}(\bar{\mathcal{X}}) = \overline{h}(\bar{\mathcal{Y}})$. Hence $\overline{h}$ descends to a well-defined smooth function $h$ on the quotient manifold $\mathcal{M}_{\rr}$ satisfying $\overline{h} = h \circ \pi$. We shall construct a family of Riemannian metrics on $\overline{\mathcal{M}}_{\rr}$ that are invariant under the equivalence relation, thereby inducing corresponding metrics on $\mathcal{M}_{\rr}$, and then apply first-order Riemannian optimization methods to solve \eqref{problem-manifold}.

\subsection{A Family of Admissible Metrics}\label{sec:metric_family}
We introduce the following class of Riemannian metrics on $\overline{\mathcal{M}}_{\rr}$. For any tuple $\va = (a_2,\dots,a_{d+1})$ with $a_k \in \{0,1,\dots,d-k+1\}$, define
\begin{equation}\label{eq:metric}
\overline{g}_{\bar{\mathcal{X}}}^{\va}(\overline{\xi},\overline{\eta})
= \sum_{k=1}^{d} \bigl\langle
\Phi(\mathcal{X}_1,\dots,\xi_k,\dots,\mathcal{X}_{k+a_{k+1}}),\;
\Phi(\mathcal{X}_1,\dots,\eta_k,\dots,\mathcal{X}_{k+a_{k+1}})
\bigr\rangle,
\end{equation}
where $\overline{\xi}=(\xi_1,\dots,\xi_d)$ and $\overline{\eta}=(\eta_1,\dots,\eta_d)$ are tangent vectors at $\bar{\mathcal{X}}$. The collection of all such metrics is denoted by $\mathfrak{M} = \{\overline{g}^{\va}\}$. For any fixed $\va$, $\overline{g}^{\va}$ defines a smooth Riemannian metric on $\overline{\mathcal{M}}_{\rr}$.

For the last core (\(k=d\)), the admissible set for \(a_{d+1}\) is \(\{0\}\); thus \(a_{d+1}=0\). Consequently, the family of metrics \(\overline{g}^{\va}\) is effectively parameterized by the first \(d-1\) entries \((a_2,\dots,a_d)\). The cardinality of the parameter set is therefore \(\prod_{k=1}^{d-1}(d-k+1)=d!\). For ease of reference, we enumerate all admissible parameter tuples in lexicographic order (with \(a_2\) as the most significant, then \(a_3\), etc.) as \(\va_1,\va_2,\dots,\va_{d!}\). In particular, \(\va_1 = (0,0,\dots,0)\) corresponds to the Euclidean metric, and \(\va_{d!} = (d-1,d-2,\dots,1,0)\) corresponds to the full metric used in previous quotient-based works. We refer to a metric by its parameter \(\va_i\) when the specific choice is relevant.

\begin{theorem}\label{lem:transmetric}
The metric \eqref{eq:metric} admits the equivalent expression
\begin{equation}\label{eq:transmetric}
\overline{g}_{\bar{\mathcal{X}}}^{\va}(\overline{\xi},\overline{\eta})
= \sum_{k=1}^{d} \bigl\langle L(\xi_k) R^{k+1},\; L(\eta_k) \bigr\rangle,
\end{equation}
where $R^{k+1} = \mathcal{X}_{\geq k+1,a_{k+1}}^{\!\top}\mathcal{X}_{\geq k+1,a_{k+1}}$.
\end{theorem}

\begin{proof}
Using the recursive relations (\ref{eq_inte}) and the left-orthogonality conditions (\ref{eq_lo}), we obtain
\begin{equation*}
		\begin{aligned}
			\overline{g}_{\bX}^{\va}(\overline{\xi},\overline{\eta})&=\sum_{k=1}^{d}\langle \Phi( \mathcal{X}_1,...,\xi_k,...,\mathcal{X}_{k+a_{k+1}}),\Phi( \mathcal{X}_1,...,\eta_k,...,\mathcal{X}_{k+a_{k+1}})\rangle
            \\&=\sum_{k=1}^{d}\langle (I_{n_k}\otimes \mathcal{X}_{\leq k-1})L(\xi_k)\mathcal{X}_{\geq k+1,a_{k+1}}^{\!\top}, (I_{n_k}\otimes \mathcal{X}_{\leq k-1})L(\eta_k)\mathcal{X}_{\geq k+1,a_{k+1}}^{\!\top}\rangle
            \\&=\sum_{k=1}^{d}\langle L(\xi_k)R^{k+1},L(\eta_k)\rangle.
		\end{aligned}
	\end{equation*}
This completes the proof.
\end{proof}

In \eqref{eq:metric}, two special choices of $\va$ are worth noting:
\begin{itemize}
\item When $\va = (0,\dots,0)$ (i.e., $\va_1$), the metric reduces to the Euclidean metric on the product space:
\[
\overline{g}_{\bar{\mathcal{X}}}^{\va}(\overline{\xi},\overline{\eta}) = \sum_{k=1}^{d} \langle \xi_k, \eta_k \rangle.
\]

\item When $\va = (d-1,d-2,\dots,0)$ (i.e., $\va_{d!}$), the metric coincides with the full metric used in previous quotient-based works:
\begin{equation}\label{maxmetric}
\overline{g}_{\bar{\mathcal{X}}}^{\va}(\overline{\xi},\overline{\eta})
= \sum_{k=1}^{d} \bigl\langle
\Phi(\mathcal{X}_1,\dots,\xi_k,\dots,\mathcal{X}_d),\;
\Phi(\mathcal{X}_1,\dots,\eta_k,\dots,\mathcal{X}_d)
\bigr\rangle.
\end{equation}
\end{itemize}

\subsection{Normal Space and Tangent Space}

Since each left unfolding $L(\mathcal{X}_k)$ lies in the Stiefel manifold $\mathrm{St}(r_k, r_{k-1}n_k)$, the tangent space of $\overline{\mathcal{M}}_{\rr}$ reduces to standard computations on a product of Stiefel manifolds~\cite{absil2008optimization}.  Following standard arguments for the Stiefel manifold ~\cite[Section~3.5.7]{absil2008optimization}, the tangent space \(T_{\bar{\mathcal{X}}}\overline{\mathcal{M}}_{\rr}\) can be written as follows.

\begin{proposition}\label{pro:tangent}
The tangent space \(T_{\bar{\mathcal{X}}}\overline{\mathcal{M}}_{\rr}\) is given by
\[
\begin{aligned}
T_{\bar{\mathcal{X}}}\overline{\mathcal{M}}_{\rr} = \bigl\{ &\bar{\eta}=(\eta_1,\dots,\eta_d) \in \prod_{k=1}^d \mathbb{R}^{r_{k-1}\times n_k\times r_k} \;\big|\;
 L(\mathcal{X}_k)^{\!\top} L(\eta_k) + L(\eta_k)^{\!\top} L(\mathcal{X}_k) = 0_{r_k}, \\
&k \in [d-1] \bigr\}.
\end{aligned}\]
\end{proposition}

For a matrix \(X\in\mathbb{R}^{m\times p}\) with full column rank, let \(\operatorname{span}(X)\) denote the subspace \(\{X\mathbf{a} : \mathbf{a}\in\mathbb{R}^p\}\). 

\begin{lemma}\label{lem:transtan}
The tangent space can also be written as
\[
\begin{aligned}
T_{\bar{\mathcal{X}}}\overline{\mathcal{M}}_{\rr} = \bigl\{ &\bar{\eta}=(\eta_1,\dots,\eta_d) \in \textstyle\prod_{k=1}^d \mathbb{R}^{r_{k-1}\times n_k\times r_k} \; \big| \;
\eta_k = L^{-1}\!\bigl(L(\mathcal{X}_k)\Omega_k + L(\mathcal{X}_k)_{\perp} B_k\bigr), \\
& \Omega_k \in \operatorname{Skew}(r_k),\; B_k \in \mathbb{R}^{(r_{k-1}n_k - r_k)\times r_k},\; k\in[d-1] \bigr\},
\end{aligned}
\]
where \(L(\mathcal{X}_k)_{\perp} \in \mathbb{R}^{r_{k-1}n_k \times (r_{k-1}n_k - r_k)}\) is any matrix whose columns form an orthonormal basis of the orthogonal complement of \(\operatorname{span}(L(\mathcal{X}_k))\).
\end{lemma}

The normal space at \(\bar{\mathcal{X}}\), denoted by \(N_{\bar{\mathcal{X}}}\overline{\mathcal{M}}_{\rr}\), is the orthogonal complement of the tangent space with respect to the metric \(\overline{g}^{\va}\) defined in \eqref{eq:transmetric}.

\begin{proposition}\label{pro:normspace}
The normal space is given by
\[
\begin{aligned}
N_{\bar{\mathcal{X}}}\overline{\mathcal{M}}_{\rr} = \bigl\{ &\bar{\xi}=(\xi_1,\dots,\xi_d) \in \textstyle\prod_{k=1}^d \mathbb{R}^{r_{k-1}\times n_k\times r_k} \; \big| \;
\xi_k = L^{-1}\!\bigl(L(\mathcal{X}_k) S_k (R^{k+1})^{-1}\bigr),\; \\
&S_k \in \operatorname{Sym}(r_k),\; k\in[d-1],\;
\xi_d = 0_{r_{d-1}\times n_d} \bigr\}.
\end{aligned}
\]
\end{proposition}

\begin{proof}
Let \(\bar{\xi} \in N_{\bar{\mathcal{X}}}\overline{\mathcal{M}}_{\rr}\). For any tangent vector \(\bar{\eta}\) in the form of
Lemma~\ref{lem:transtan},
the orthogonality condition \(\overline{g}^{\va}_{\bar{\mathcal{X}}}(\bar{\xi},\bar{\eta})=0\) becomes
\[
\sum_{k=1}^{d-1} \bigl\langle L(\xi_k), \bigl(L(\mathcal{X}_k)\Omega_k + L(\mathcal{X}_k)_{\perp} B_k\bigr) R^{k+1} \bigr\rangle + \langle L(\xi_d), L(\eta_d) \rangle = 0.
\]
Since \(\eta_d\) is arbitrary, we have \(\xi_d = 0_{r_{d-1}\times n_d}\). For each \(k\in[d-1]\), writing
\[
L(\xi_k)R^{k+1} = L(\mathcal{X}_k) S_k + L(\mathcal{X}_k)_{\perp} D_k,
\]
with \(S_k \in \mathbb{R}^{r_k\times r_k}\) and \(D_k \in \mathbb{R}^{(r_{k-1}n_k - r_k)\times r_k}\), the orthogonality condition reduces to
\[
\operatorname{tr}\bigl( \Omega_k^{\!\top} S_k + B_k^{\!\top} D_k \bigr) = 0.
\]
Choosing first \(B_k = 0_{(r_{k-1}n_k - r_k)\times r_k}\) shows that \(S_k\) must be symmetric; then choosing \(\Omega_k = 0_{r_k}\) forces \(D_k = 0_{(r_{k-1}n_k - r_k)\times r_k}\). Consequently \(L(\xi_k) = L(\mathcal{X}_k) S_k (R^{k+1})^{-1}\) with \(S_k\) symmetric.
\end{proof}

With the tangent and normal spaces characterized, we now give an explicit formula for the orthogonal projection onto the tangent space with respect to the metric \(\overline{g}^{\va}\).

\begin{theorem}\label{tproj}
The orthogonal projection \(\psi_{\bar{\mathcal{X}}} : \prod_{k=1}^d \mathbb{R}^{r_{k-1}\times n_k\times r_k} \to T_{\bar{\mathcal{X}}}\overline{\mathcal{M}}_{\rr}\) associated with the metric \(\overline{g}^{\va}\) is given by
\begin{equation}\label{op1}
\psi_{\bar{\mathcal{X}}}(\bar{\zeta}) = \bigl( \zeta_1 - \xi_1,\; \dots,\; \zeta_{d-1} - \xi_{d-1},\; \zeta_d \bigr),
\end{equation}
where for each \(k\in[d-1]\),
\[
\xi_k = L^{-1}\!\bigl( L(\mathcal{X}_k) S_k (R^{k+1})^{-1} \bigr),
\]
and \(S_k\) is the unique solution of the Lyapunov equation
\begin{equation}\label{eq:lyap}
(R^{k+1})^{-1} S_k + S_k (R^{k+1})^{-1} = L(\mathcal{X}_k)^{\!\top} L(\zeta_k) + L(\zeta_k)^{\!\top} L(\mathcal{X}_k).
\end{equation}
\end{theorem}

\begin{proof}
Let \(\bar{\eta} = \psi_{\bar{\mathcal{X}}}(\bar{\zeta})\). By Proposition~\ref{pro:normspace}, we can write \(\eta_k = \zeta_k - \xi_k\) with \(\xi_k\) of the stated form. The condition that \(\bar{\eta}\) belongs to the tangent space, i.e. \(L(\mathcal{X}_k)^{\!\top}L(\eta_k) + L(\eta_k)^{\!\top}L(\mathcal{X}_k)=0_{r_k}\), then becomes
\[
L(\mathcal{X}_k)^{\!\top}L(\zeta_k) + L(\zeta_k)^{\!\top}L(\mathcal{X}_k) - (R^{k+1})^{-1}S_k - S_k(R^{k+1})^{-1} = 0_{r_k},
\]
which is exactly equation (\ref{eq:lyap}).
\end{proof}

\begin{remark}
When the parameter \(a_{k+1}=0\), we have \(R^{k+1}=I_{r_k}\) and the Lyapunov equation \eqref{eq:lyap} reduces to \(S_k = \frac{1}{2}\bigl(L(\mathcal{X}_k)^{\!\top}L(\zeta_k)+L(\zeta_k)^{\!\top}L(\mathcal{X}_k)\bigr)\). In this case the projection can be computed without solving a Lyapunov equation.
\end{remark}

\subsection{Vertical Space and Horizontal Space}
The tangent space to the equivalence class \([\bar{\mathcal{X}}]\) at \(\bar{\mathcal{X}}\) is called the vertical space at \(\bar{\mathcal{X}}\) and is denoted by \(\mathcal{V}_{\bar{\mathcal{X}}}\).

\begin{proposition}\label{pro:vcharacter}
The vertical space \(\mathcal{V}_{\bar{\mathcal{X}}}\) is given by
\[
\mathcal{V}_{\bar{\mathcal{X}}} = \bigl\{
\bigl(\mathcal{X}_1 \times_3 \Omega_1,\;
\mathcal{X}_2 \times_1 \Omega_1 + \mathcal{X}_2 \times_3 \Omega_2,\;
\dots,\;
\mathcal{X}_d \times_1 \Omega_{d-1}\bigr)
\; \big| \;
\Omega_k \in \operatorname{Skew}(r_k),\; k \in [d-1] \bigr\}.
\]
\end{proposition}

\begin{proof}
Consider a curve \(t \mapsto \bar{\mathcal{X}}(t)\) in the equivalence class \([\bar{\mathcal{X}}]\) passing through \(\bar{\mathcal{X}}\) at \(t=0\). By the characterization of the equivalence relation, there exist smooth curves \(Q_k(t) \in \mathrm{O}(r_k)\) with \(Q_k(0)=I_{r_k}\) such that
\[
\bar{\mathcal{X}}(t) = \bigl(\mathcal{X}_1 \times_3 Q_1(t),\;
\mathcal{X}_2 \times_1 Q_1(t) \times_3 Q_2(t),\;
\dots,\;
\mathcal{X}_d \times_1 Q_{d-1}(t)\bigr).
\]
Differentiating the \(k\)-th core at \(t=0\) gives
\[
\dot{\mathcal{X}}_k(0) = \mathcal{X}_k \times_1 \dot{Q}_{k-1}(0) + \mathcal{X}_k \times_3 \dot{Q}_k(0).
\]
 Differentiating the orthogonality condition \(Q_k(t)Q_k(t)^{\!\top}=I_{r_k}\) at \(t=0\) yields \(\dot{Q}_k(0) + \dot{Q}_k(0)^{\!\top}=0_{r_k}\); hence \(\dot{Q}_k(0)\) is skew-symmetric. Writing \(\Omega_k = \dot{Q}_k(0)\) completes the proof.
\end{proof}

The horizontal space \(\mathcal{H}_{\bar{\mathcal{X}}}\) is defined as the orthogonal complement of the vertical space within the tangent space \(T_{\bar{\mathcal{X}}}\overline{\mathcal{M}}_{\rr}\) with respect to the metric \(\overline{g}^{\va}\).

For a tangent vector \(\xi \in T_{[\bar{\mathcal{X}}]}\mathcal{M}_{\rr}\) on the quotient manifold, any element \(\bar{\xi}_{\bar{\mathcal{X}}} \in T_{\bar{\mathcal{X}}}\overline{\mathcal{M}}_{\rr}\) satisfying \(D\pi(\bar{\mathcal{X}})[\bar{\xi}_{\bar{\mathcal{X}}}] = \xi\) is called a lift of \(\xi\). Among all such lifts, the unique one that belongs to the horizontal space is termed the horizontal lift.

\begin{proposition}\label{hcharacter}
With respect to the metric \(\overline{g}^{\va}\), the horizontal space admits the characterization
\[
\begin{aligned}
\mathcal{H}_{\bar{\mathcal{X}}} = \bigl\{ &\bar{\xi}=(\xi_1,\dots,\xi_d) \in T_{\bar{\mathcal{X}}}\overline{\mathcal{M}}_{\rr} \; \big| \; L(\mathcal{X}_k)^{\!\top} L(\xi_k) R^{k+1} \\
&- R(\xi_{k+1}) (R^{k+2} \otimes I_{n_{k+1}}) R(\mathcal{X}_{k+1})^{\!\top}
\in \operatorname{Sym}(r_k),\; k \in [d-1] \bigr\}.
\end{aligned}
\]
\end{proposition}

\begin{proof}
Let \(\bar{\xi} \in \mathcal{H}_{\bar{\mathcal{X}}}\) and take an arbitrary vertical vector
\[
\bar{\eta} = \bigl(\mathcal{X}_1 \times_3 \Omega_1,\; \mathcal{X}_2 \times_1 \Omega_1 + \mathcal{X}_2 \times_3 \Omega_2,\; \dots,\; \mathcal{X}_d \times_1 \Omega_{d-1}\bigr)
\]
with \(\Omega_k\) skew-symmetric. Orthogonality \(\overline{g}^{\va}_{\bar{\mathcal{X}}}(\bar{\xi},\bar{\eta})=0\) must hold. Using the expression (\ref{eq:transmetric}) for the metric, we compute for each \(k\)
\begin{equation*}
		\begin{aligned}
			&\langle \Phi( \mathcal{X}_1,...,\xi_k,...,\mathcal{X}_{k+a_{k+1}}),\Phi( \mathcal{X}_1,...,\mathcal{X}_k\times _1\Omega_{k-1},...,\mathcal{X}_{k+a_{k+1}})\rangle
		\\&=	\langle \Phi( \mathcal{X}_1,...,\xi_k,...,\mathcal{X}_{k+a_{k+1}}),\Phi( \mathcal{X}_1,...,\mathcal{X}_{k-1}\times _3\Omega_{k-1}^{\!\top},...,\mathcal{X}_{k+a_{k+1}})\rangle
		\\&=\langle \mathcal{X}_{\leq k-1}R(\xi_k)(\mathcal{X}_{\geq k+1,a_{k+1}}^{\!\top}\otimes I_{n_k}), \mathcal{X}_{\leq k-1}\Omega_{k-1}R(\mathcal{X}_k)(\mathcal{X}_{\geq k+1,a_{k+1}}^{\!\top}\otimes I_{n_k})\rangle
		\\&=\langle R(\xi_k)(R^{k+1}\otimes I_{n_k}) R(\mathcal{X}_k)^{\!\top}, \Omega_{k-1}\rangle ,
		\end{aligned}
	\end{equation*} and
	\begin{equation*}
		\begin{aligned}
			&\langle \Phi( \mathcal{X}_1,...,\xi_k,...,\mathcal{X}_{k+a_{k+1}}),\Phi( \mathcal{X}_1,...,\mathcal{X}_k\times _3\Omega_k,...,\mathcal{X}_{k+a_{k+1}})\rangle
		\\&=	\langle (I_{n_k}\otimes \mathcal{X}_{\leq k-1})L(\xi_k)\mathcal{X}_{\geq k+1,a_{k+1}}^{\!\top}, (I_{n_k}\otimes \mathcal{X}_{\leq k-1})L(\mathcal{X}_k)\Omega_k^{\!\top}\mathcal{X}_{\geq k+1,a_{k+1}}^{\!\top}\rangle
		\\&=-\langle L(\mathcal{X}_k)^{\!\top}L(\xi_k)R^{k+1},\Omega_k \rangle.
		\end{aligned}
	\end{equation*}
	Summing over \(k\) and using the skew-symmetry of \(\Omega_k\) gives
\begin{align*}
       \overline{g}_{\bar{\mathcal{X}}}^{\va}(\overline{\xi}, \overline{\eta})
        &= \sum_{k=1}^{d-1} \Bigl( \bigl\langle R(\xi_{k+1}) (R^{k+2} \otimes I_{n_{k+1}}) R(\mathcal{X}_{k+1})^{\!\top}, \; \Omega_{k} \bigr\rangle - \bigl\langle L(\mathcal{X}_k)^{\!\top} L(\xi_k) R^{k+1}, \; \Omega_k \bigr\rangle \Bigr) \\
        &= \sum_{k=1}^{d-1} \bigl\langle R(\xi_{k+1}) (R^{k+2} \otimes I_{n_{k+1}}) R(\mathcal{X}_{k+1})^{\!\top} - L(\mathcal{X}_k)^{\!\top} L(\xi_k) R^{k+1}, \; \Omega_k \bigr\rangle .
\end{align*}
Since the \(\Omega_k\) are arbitrary skew-symmetric matrices, the matrix inside the inner product must be symmetric for each \(k\), which yields the stated condition.
\end{proof}

To compute the horizontal projection explicitly we solve a coupled system of Lyapunov equations, analogous to those appearing in \cite[Proposition 3]{kasai2016low}. Introducing a suitable matrix \(\Delta_r\) reduces the problem to a linear system with a block-tridiagonal coefficient matrix. This formulation gives a solution to the linear system and can be used to analyze the projection's properties.

Define the matrix \(\Delta_r \in \mathbb{R}^{\frac{r(r-1)}{2} \times r^2}\) as follows: for each pair \((i,j)\) with \(i>j\) and \(i,j\in[r]\), set
\begin{align*}
    &\Delta_r\!\Bigl(\frac{(i-1)(i-2)}{2}+j,\; i+r(j-1)\Bigr) = 1,\\
    &\Delta_r\!\Bigl(\frac{(i-1)(i-2)}{2}+j,\; j+r(i-1)\Bigr) = -1,
\end{align*}
and all other entries zero. From the construction, \(\Delta_r\) has full row rank.

\begin{lemma}\label{ldelta}
Let \(A \in \mathbb{R}^{r\times r}\) and  \(\mathbf{c}\) be the column vector obtained by stacking the entries of \(A-A^{\!\top}\) for indices \(i>j\) in column-major order. Then \(\mathbf{c} = \Delta_r \operatorname{vec}(A)\).
\end{lemma}

\begin{proof}
Let \(C = A - A^{\!\top}\). The entry \(C_{ij}\) with \(i>j\) is located at position \(\frac{(i-1)(i-2)}{2}+j\) in \(\mathbf{c}\). In the vectorized representation, the entry \(A_{ij}\) corresponds to the \((i+r(j-1))\)-th element of \(\operatorname{vec}(A)\), while \(A_{ji}\) corresponds to the \((j+r(i-1))\)-th element. Hence
\[
\begin{aligned}
\mathbf{c}\!\Bigl(\frac{(i-1)(i-2)}{2}+j\Bigr)
&= \operatorname{vec}(A)_{i+r(j-1)} - \operatorname{vec}(A)_{j+r(i-1)} \\
&= \Delta_r\!\Bigl(\frac{(i-1)(i-2)}{2}+j,\;:\Bigr) \operatorname{vec}(A),
\end{aligned}
\]
which is exactly the statement \(\mathbf{c} = \Delta_r \operatorname{vec}(A)\).
\end{proof}

\begin{lemma}\label{rdelta}
Let \(\Omega \in \mathbb{R}^{r\times r}\) be skew-symmetric and  \(\omega\) be the column vector of its strictly lower-triangular entries taken in column-major order. Then \(\operatorname{vec}(\Omega) = \Delta_r^{\!\top} \omega\).
\end{lemma}

\begin{proof}
The proof follows directly from the definition of \(\Delta_r\): the matrix \(\Delta_r^{\!\top}\) maps a vector \(\omega\) of length \(\frac{r(r-1)}{2}\) to a vectorized \(r\times r\) matrix whose strictly lower-triangular part equals \(\omega\), the strictly uppertriangular part equals \(-\omega\), and the diagonal is zero, i.e., exactly to \(\operatorname{vec}(\Omega)\).
\end{proof}

With these lemmas we now give an explicit formula for the horizontal projection.

\begin{theorem}\label{hproj}
The orthogonal projection onto the horizontal space \(\mathcal{H}_{\bar{\mathcal{X}}}\) with respect to the metric \(\overline{g}^{\va}\) is given by
\[
\phi_{\bar{\mathcal{X}}}(\bar{\xi}) =
\bigl( \xi_1 - \mathcal{X}_1 \times_3 \Omega_1,\;
\xi_2 - \mathcal{X}_2 \times_1 \Omega_1 - \mathcal{X}_2 \times_3 \Omega_2,\;
\dots,\;
\xi_d - \mathcal{X}_d \times_1 \Omega_{d-1} \bigr),
\]
where \(\bar{\xi}=(\xi_1,\dots,\xi_d) \in T_{\bar{\mathcal{X}}}\overline{\mathcal{M}}_{\rr}\) and \(\Omega_k\) are skew-symmetric matrices determined as follows. Let \(\omega_k\) be the column vector of the strictly lower-triangular entries of \(\Omega_k\) in column-major order. Then the vector \(\omega = (\omega_1,\dots,\omega_{d-1})\) satisfies the linear system
\begin{equation}\label{eq:hmetrix}
\begin{bmatrix}
A_1 & B_1 & & \\
B_1^{\!\top} & A_2 & B_2 & \\
& \ddots & \ddots & \ddots \\
& & B_{d-2}^{\!\top} & A_{d-1}
\end{bmatrix}
\begin{bmatrix}
\omega_1 \\ \omega_2 \\ \vdots \\ \omega_{d-1}
\end{bmatrix}
=
\begin{bmatrix}
\mathbf{c}_1 \\ \mathbf{c}_2 \\ \vdots \\ \mathbf{c}_{d-1}
\end{bmatrix},
\end{equation}
with blocks
\begin{align*}
A_k &= \Delta_{r_k} \bigl( (R^{k+1} + R(\mathcal{X}_{k+1})(R^{k+2} \otimes I_{n_{k+1}}) R(\mathcal{X}_{k+1})^{\!\top}) \otimes I_{r_k} \bigr) \Delta_{r_k}^{\!\top}, \\
B_k &= -\Delta_{r_k} \bigl( R(\mathcal{X}_{k+1}) \otimes I_{r_k} \bigr) \bigl( R^{k+2} \otimes L(\mathcal{X}_{k+1}) \bigr) \Delta_{r_{k+1}}^{\!\top}, \\
\mathbf{c}_k &= \Delta_{r_k} \bigl( -(R^{k+1} \otimes L(\mathcal{X}_k)^{\!\top}) \operatorname{vec}(\xi_k)
+ (R(\mathcal{X}_{k+1}) \otimes I_{r_k}) \operatorname{vec}(\xi_{k+1} \times_3 R^{k+2}) \bigr).
\end{align*}
\end{theorem}

\begin{proof}
From Proposition~\ref{hcharacter}, the projected vector \(\phi_{\bar{\mathcal{X}}}(\bar{\xi}) = \bar{\xi} - \bar{\eta}\) with \(\bar{\eta} \in \mathcal{V}_{\bar{\mathcal{X}}}\) satisfies that for each \(k\in[d-1]\) the matrix
\[
D_k := L(\mathcal{X}_k)^{\!\top} L(\xi_k-\eta_k) R^{k+1}
- R(\xi_{k+1}-\eta_{k+1}) (R^{k+2} \otimes I_{n_{k+1}}) R(\mathcal{X}_{k+1})^{\!\top}
\]
is symmetric. Substituting the expression for \(\eta_k\) (see Proposition~\ref{pro:vcharacter}) and using the unfolding relations, vectorization gives
\begin{align*}
\operatorname{vec}(D_k) ={}&
\bigl( (R^{k+1}+R(\mathcal{X}_{k+1})(R^{k+2}\otimes I_{n_{k+1}})R(\mathcal{X}_{k+1})^{\!\top}) \otimes I_{r_k} \bigr) \operatorname{vec}(\Omega_k) \\
& - (R^{k+1}\otimes L(\mathcal{X}_k)^{\!\top}) (R(\mathcal{X}_k)^{\!\top}\otimes I_{r_{k-1}}) \operatorname{vec}(\Omega_{k-1}) \\
& - (R(\mathcal{X}_{k+1})\otimes I_{r_k}) (R^{k+2}\otimes L(\mathcal{X}_{k+1})) \operatorname{vec}(\Omega_{k+1}) \\
& + (R^{k+1}\otimes L(\mathcal{X}_k)^{\!\top}) \operatorname{vec}(\xi_k)
- (R(\mathcal{X}_{k+1})\otimes I_{r_k}) \operatorname{vec}(\xi_{k+1}\times_3 R^{k+2}).
\end{align*}
The symmetry condition \(D_k - D_k^{\!\top}=0\) is equivalent, via Lemma~\ref{ldelta}, to \(\Delta_{r_k} \operatorname{vec}(D_k)=0\). Applying Lemmas~\ref{ldelta} and \ref{rdelta} then yields
\[
A_k \omega_k + B_{k-1}^{\!\top} \omega_{k-1} + B_k \omega_{k+1} = \mathbf{c}_k,
\]
with the convention \(\omega_0=\omega_d=0\). Collecting these equations for \(k=1,\dots,d-1\) gives the block-tridiagonal system (\ref{eq:hmetrix}).
\end{proof}

Observe from Theorem~\ref{hproj} that, although both our work and \cite{cai2022tensor} require solving a linear system for the horizontal projection, the dimension of our system is smaller because we solve for the \(\frac{r_k(r_k-1)}{2}\) independent entries of each skew-symmetric matrix \(\Omega_k\), whereas in \cite{cai2022tensor} the analogous matrices are full square matrices of size \(r_k\). Moreover, the cost of forming the coefficient matrices is \(\mathcal{O}(d n r^5)\) in our setting, compared to \(\mathcal{O}(d n r^6)\) in \cite{cai2022tensor}.

\subsection{Riemannian Metric on the Quotient Manifold}\label{rieme}
Having introduced the vertical and horizontal spaces, we now verify that the family of metrics \(\overline{g}^{\va}\) defined in \eqref{eq:transmetric} descends to a well-defined Riemannian metric on the quotient manifold \(\mathcal{M}_{\rr}\). As established in the introduction, the choice of the orthogonal group as the quotient group---combined with the left-orthogonal structure---enables the construction of this family. The key step is to confirm that the action of the orthogonal group defining the equivalence classes is isometric, i.e., it preserves each metric \(\overline{g}^{\va}\). This is precisely the condition that allows the metrics to induce a Riemannian metric on \(\mathcal{M}_{\rr}\).

\begin{lemma}\label{hlift}
Let \(\xi_{[\bar{\mathcal{X}}]}\in T_{[\bar{\mathcal{X}}]}\mathcal{M}_{\rr}\) and let \(\bar{\mathcal{X}},\bar{\mathcal{Y}}\in[\bar{\mathcal{X}}]\) be two representatives of the same equivalence class, with \(\bar{\mathcal{Y}}=\varphi_Q(\bar{\mathcal{X}})\) as in \eqref{orth-group}. Let \(\overline{\xi}_{\bar{\mathcal{X}}}\) and \(\overline{\xi}_{\bar{\mathcal{Y}}}\) be the horizontal lifts of \(\xi_{[\bar{\mathcal{X}}]}\) at \(\bar{\mathcal{X}}\) and \(\bar{\mathcal{Y}}\), respectively. Then
\[
\overline{\xi}_{\bar{\mathcal{Y}}} = \varphi_Q(\overline{\xi}_{\bar{\mathcal{X}}}).
\]
\end{lemma}

\begin{proof}
Let \(\xi_{[\bar{\mathcal{X}}]}\in T_{[\bar{\mathcal{X}}]}\mathcal{M}_{\rr}\) and let \(\bar{\mathcal{X}},\bar{\mathcal{Y}}\in[\bar{\mathcal{X}}]\) with \(\bar{\mathcal{Y}}=\varphi_Q(\bar{\mathcal{X}})\). By the definition of the horizontal lift, \(\overline{\xi}_{\bar{\mathcal{X}}}\) is the unique horizontal vector satisfying \(D\pi(\bar{\mathcal{X}})[\overline{\xi}_{\bar{\mathcal{X}}}] = \xi_{[\bar{\mathcal{X}}]}\). Since \(\pi\circ\varphi_Q = \pi\) and \(\varphi_Q\) is linear,
\[
D\pi(\bar{\mathcal{Y}})[\varphi_Q(\overline{\xi}_{\bar{\mathcal{X}}})] = D\pi(\bar{\mathcal{Y}})\circ D\varphi_Q(\bar{\mathcal{X}})[\overline{\xi}_{\bar{\mathcal{X}}}] = D(\pi\circ\varphi_Q)(\bar{\mathcal{X}})[\overline{\xi}_{\bar{\mathcal{X}}}] = D\pi(\bar{\mathcal{X}})[\overline{\xi}_{\bar{\mathcal{X}}}] = \xi_{[\bar{\mathcal{X}}]}.
\]
Thus \(\varphi_Q(\overline{\xi}_{\bar{\mathcal{X}}})\) is a lift of \(\xi_{[\bar{\mathcal{X}}]}\) at \(\bar{\mathcal{Y}}\). It remains to verify that this lift is horizontal, i.e., \(\varphi_Q(\overline{\xi}_{\bar{\mathcal{X}}})\in\mathcal{H}_{\bar{\mathcal{Y}}}\).

Write \(\overline{\xi}_{\bar{\mathcal{X}}}=(\xi_1^x,\dots,\xi_d^x)\) and set \(\varphi_Q(\overline{\xi}_{\bar{\mathcal{X}}})=(\xi_1^y,\dots,\xi_d^y)\). We need to verify that \(\varphi_Q(\overline{\xi}_{\bar{\mathcal{X}}})\) satisfies the horizontal condition in Proposition~\ref{hcharacter}. Recall that for a representative \(\bar{\mathcal{X}}\), the matrix \(R^{k+1,x}\) is defined as \(R^{k+1,x} = \mathcal{X}_{\ge k+1,a_{k+1}}^{\!\top}\mathcal{X}_{\ge k+1,a_{k+1}}\). Under the transformation \(\bar{\mathcal{Y}}=\varphi_Q(\bar{\mathcal{X}})\), the corresponding matrix transforms as
\begin{equation}\label{eq:R-transform}
R^{k+1,y} = \mathcal{Y}_{\ge k+1,a_{k+1}}^{\!\top}\mathcal{Y}_{\ge k+1,a_{k+1}} = Q_k R^{k+1,x} Q_k^{\!\top},
\end{equation}
because \(\mathcal{Y}_{\ge k+1,a_{k+1}} = (Q_{k+a_{k+1}}\otimes I_{\prod_{i=k+1}^{k+a_{k+1}} n_i})\mathcal{X}_{\ge k+1,a_{k+1}} Q_k^{\!\top}\).

 Now, for each \(k\in[d-1]\), using the transformation rules (\ref{eq:Ltrans})--(\ref{eq:Rtrans}) for the unfoldings and the relation \eqref{eq:R-transform}, we compute:
\begin{align*}
&L(\mathcal{Y}_k)^{\!\top}L(\xi_k^y)R^{k+1,y} - R(\xi_{k+1}^y)(R^{k+2,y}\otimes I_{n_{k+1}})R(\mathcal{Y}_{k+1})^{\!\top} \\
&= Q_k L(\mathcal{X}_k)^{\!\top}(I_{n_k}\otimes Q_{k-1})^{\!\top}(I_{n_k}\otimes Q_{k-1}) L(\xi_k^x) Q_k^{\!\top} \bigl(Q_k R^{k+1,x} Q_k^{\!\top}\bigr)- Q_k R(\xi_{k+1}^x) \\
&\quad (Q_{k+1}\otimes I_{n_{k+1}})^{\!\top}\bigl((Q_{k+1}R^{k+2,x}Q_{k+1}^{\!\top}) \otimes I_{n_{k+1}}\bigr) (Q_{k+1}\otimes I_{n_{k+1}})R(\mathcal{X}_{k+1})^{\!\top} Q_k^{\!\top} \\
&= Q_k \bigl(L(\mathcal{X}_k)^{\!\top}L(\xi_k^x)R^{k+1,x} - R(\xi_{k+1}^x)(R^{k+2,x}\otimes I_{n_{k+1}})R(\mathcal{X}_{k+1})^{\!\top}\bigr) Q_k^{\!\top}.
\end{align*}

Because \(\overline{\xi}_{\bar{\mathcal{X}}}\) is horizontal, the matrix inside the parentheses is symmetric. Conjugating a symmetric matrix by an orthogonal matrix \(Q_k\) preserves symmetry. Therefore \(\varphi_Q(\overline{\xi}_{\bar{\mathcal{X}}})\) satisfies the horizontal condition of Proposition~\ref{hcharacter}, i.e., \(\varphi_Q(\overline{\xi}_{\bar{\mathcal{X}}})\in\mathcal{H}_{\bar{\mathcal{Y}}}\). Consequently, \[
\overline{\xi}_{\bar{\mathcal{Y}}} = \varphi_Q(\overline{\xi}_{\bar{\mathcal{X}}}),
\]which completes the proof.
\end{proof}

\begin{theorem}[Invariance of the metric]\label{the:metric}
Let \(\bar{\mathcal{X}},\bar{\mathcal{Y}}\in[\bar{\mathcal{X}}]\) and let \(\overline{\xi}_{\bar{\mathcal{X}}},\overline{\xi}_{\bar{\mathcal{Y}}}\) (respectively \(\overline{\eta}_{\bar{\mathcal{X}}},\overline{\eta}_{\bar{\mathcal{Y}}}\)) be the horizontal lifts of the same tangent vector \(\xi_{[\bar{\mathcal{X}}]}\) (respectively \(\eta_{[\bar{\mathcal{X}}]}\)) at the two points. Then
\[
\overline{g}^{\va}_{\bar{\mathcal{Y}}}(\overline{\xi}_{\bar{\mathcal{Y}}},\overline{\eta}_{\bar{\mathcal{Y}}}) = \overline{g}^{\va}_{\bar{\mathcal{X}}}(\overline{\xi}_{\bar{\mathcal{X}}},\overline{\eta}_{\bar{\mathcal{X}}}).
\]
\end{theorem}

\begin{proof}
Let \(\bar{\mathcal{Y}}=\varphi_Q(\bar{\mathcal{X}})\). By Lemma~\ref{hlift}, \(\overline{\xi}_{\bar{\mathcal{Y}}}=\varphi_Q(\overline{\xi}_{\bar{\mathcal{X}}})\) and \(\overline{\eta}_{\bar{\mathcal{Y}}}=\varphi_Q(\overline{\eta}_{\bar{\mathcal{X}}})\). From \eqref{eq:transmetric},
\[
\overline{g}^{\va}_{\bar{\mathcal{Y}}}(\overline{\xi}_{\bar{\mathcal{Y}}},\overline{\eta}_{\bar{\mathcal{Y}}})
= \sum_{k=1}^{d} \bigl\langle L(\xi_k^y) R^{k+1,y},\; L(\eta_k^y) \bigr\rangle.
\]
Using the unfolding rules \eqref{eq:Ltrans} and the definition of \(\varphi_Q\),
\(L(\xi_k^y) = (I_{n_k}\otimes Q_{k-1})L(\xi_k^x)Q_k^{\!\top}\), and similarly for \(L(\eta_k^y)\).

Inserting these together with \(R^{k+1,y} = Q_k R^{k+1,x} Q_k^{\!\top}\) from \eqref{eq:R-transform} gives
\begin{align*}
\overline{g}^{\va}_{\bar{\mathcal{Y}}}(\overline{\xi}_{\bar{\mathcal{Y}}},\overline{\eta}_{\bar{\mathcal{Y}}})
&= \sum_{k=1}^{d} \bigl\langle (I_{n_k}\otimes Q_{k-1})L(\xi_k^x)Q_k^{\!\top} Q_k R^{k+1,x} Q_k^{\!\top},\; (I_{n_k}\otimes Q_{k-1})L(\eta_k^x)Q_k^{\!\top} \bigr\rangle \\
&= \sum_{k=1}^{d} \bigl\langle L(\xi_k^x) R^{k+1,x},\; L(\eta_k^x) \bigr\rangle
= \overline{g}^{\va}_{\bar{\mathcal{X}}}(\overline{\xi}_{\bar{\mathcal{X}}},\overline{\eta}_{\bar{\mathcal{X}}}),
\end{align*}
which completes the proof.
\end{proof}

Because the metric $\overline{g}^{\va}$ is invariant under the action of the orthogonal group, it induces a unique Riemannian metric $g^{\va}$ on the quotient manifold $\mathcal{M}_{\rr}$ via the standard quotient construction \cite[Section 3.6.2]{absil2008optimization}, \cite[Section 9.7]{boumal2023introduction}:
\[
g^{\va}_{[\bar{\mathcal{X}}]}(\xi_{[\bar{\mathcal{X}}]},\eta_{[\bar{\mathcal{X}}]})
:= \overline{g}^{\va}_{\bar{\mathcal{X}}}(\overline{\xi}_{\bar{\mathcal{X}}},\overline{\eta}_{\bar{\mathcal{X}}}),
\]
where $\overline{\xi}_{\bar{\mathcal{X}}},\overline{\eta}_{\bar{\mathcal{X}}}$ are the horizontal lifts of $\xi_{[\bar{\mathcal{X}}]},\eta_{[\bar{\mathcal{X}}]}$ at an arbitrary representative $\bar{\mathcal{X}}\in[\bar{\mathcal{X}}]$.

The orthogonality of the matrices $Q_k$ is essential for the above invariance. If, as in \cite{cai2022tensor}, the transformations were allowed to be arbitrary invertible matrices ($Q_k\in\mathrm{GL}(r_k)$), the inner product would not be preserved unless the parameter tuple $\va$ is restricted to the special case $\va=(d-1,d-2,\dots,0)$ (the full metric). The left-orthogonal structure of $\overline{\mathcal{M}}_{\rr}$ therefore provides a larger family of admissible metrics on the quotient, which highlights an advantage of our geometric setting.

\subsection{Riemannian Gradient}

The Riemannian gradient of the function \(\overline{h}\) at a point \(\bar{\mathcal{X}}\in\overline{\mathcal{M}}_{\rr}\), denoted by \(\operatorname{grad} \overline{h}(\bar{\mathcal{X}})\), is the unique tangent vector in \(T_{\bar{\mathcal{X}}}\overline{\mathcal{M}}_{\rr}\) satisfying
\[
\overline{g}^{\va}_{\bar{\mathcal{X}}}\!\bigl(\operatorname{grad} \overline{h}(\bar{\mathcal{X}}),\ \overline{\xi}\bigr)=D\overline{h}(\bar{\mathcal{X}})[\overline{\xi}], \qquad \forall\ \overline{\xi}\in T_{\bar{\mathcal{X}}}\overline{\mathcal{M}}_{\rr}.
\]

\begin{proposition}\label{pro:rgrad}
The Riemannian gradient of \(\overline{h}\) at \(\bar{\mathcal{X}}\) is given by
\[
\operatorname{grad} \overline{h}(\bar{\mathcal{X}}) = \psi_{\bar{\mathcal{X}}}\!\bigl(\operatorname{egrad}\overline{h}(\bar{\mathcal{X}})\bigr),
\]
where \(\psi_{\bar{\mathcal{X}}}\) is the orthogonal projection defined by \eqref{op1} and the Euclidean gradient \(\operatorname{egrad}\overline{h}(\bar{\mathcal{X}})\) has components
\[
\operatorname{egrad}\overline{h}(\bar{\mathcal{X}}) = \Bigl\{L^{-1}\!\Bigl((I_{n_k}\otimes \mathcal{X}_{\leq k-1}^{\!\top})\,
\bigl[\mathcal{P}_{\Omega}(\Phi(\bar{\mathcal{X}}))-\mathcal{P}_{\Omega}(\Gamma)\bigr]_{\langle k \rangle}\,
\mathcal{X}_{\geq k+1}\, (R^{k+1})^{-1}\Bigr)\Bigr\}_{k=1}^{d}.
\]
\end{proposition}

\begin{proof}
For a tangent vector \(\overline{\xi}=(\xi_1,\dots,\xi_d)\in T_{\bar{\mathcal{X}}}\overline{\mathcal{M}}_{\rr}\),
\begin{align*}
D\overline{h}(\bar{\mathcal{X}})[\overline{\xi}]
&= \lim_{t\to 0}\frac{\overline{h}(\bar{\mathcal{X}}+t\overline{\xi})-\overline{h}(\bar{\mathcal{X}})}{t} \\
&= \sum_{k=1}^{d} \bigl\langle \bigl[\mathcal{P}_{\Omega}(\Phi(\bar{\mathcal{X}}))-\mathcal{P}_{\Omega}(\Gamma)\bigr]_{\langle k \rangle},\;
(I_{n_k}\otimes \mathcal{X}_{\leq k-1}) L(\xi_k) \mathcal{X}_{\geq k+1}^{\!\top} \bigr\rangle.
\end{align*}
Using the recursive relation \(\mathcal{X}_{\geq k+1} = (\mathcal{X}_{\geq k+1+a_{k+1}}\otimes I_{n_{k+1}\times\cdots\times n_{k+a_{k+1}}})\,\mathcal{X}_{\geq k+1,a_{k+1}}\) and the definition \(R^{k+1}=\mathcal{X}_{\geq k+1,a_{k+1}}^{\!\top}\mathcal{X}_{\geq k+1,a_{k+1}}\), we rewrite the inner product as

\begin{multline*}
\bigl\langle (I_{n_k}\otimes \mathcal{X}_{\leq k-1}^{\!\top})\,
\bigl[\mathcal{P}_{\Omega}(\Phi(\bar{\mathcal{X}}))-\mathcal{P}_{\Omega}(\Gamma)\bigr]_{\langle k \rangle}\,
\mathcal{X}_{\geq k+1}\,(R^{k+1})^{-1}\,\mathcal{X}_{\geq k+1,a_{k+1}}^{\!\top},\\
L(\xi_k)\,\mathcal{X}_{\geq k+1,a_{k+1}}^{\!\top} \bigr\rangle.
\end{multline*}

Summing over \(k\) and comparing with the metric expression (\ref{eq:transmetric}) yields
\[
D\overline{h}(\bar{\mathcal{X}})[\overline{\xi}] = \overline{g}^{\va}_{\bar{\mathcal{X}}}\!\bigl(\operatorname{egrad}\overline{h}(\bar{\mathcal{X}}),\ \overline{\xi}\bigr).
\]
Thus \(\operatorname{egrad}\overline{h}(\bar{\mathcal{X}})\) satisfies the defining relation of the Riemannian gradient in the ambient Euclidean space. Projecting it onto the tangent space via \(\psi_{\bar{\mathcal{X}}}\) gives the Riemannian gradient on \(\overline{\mathcal{M}}_{\rr}\).
\end{proof}

\subsection{Retraction}\label{sec:retraction}

A retraction on a manifold $\overline{\mathcal{M}}$ is a smooth mapping $\overline{R}:T\overline{\mathcal{M}}\to\overline{\mathcal{M}}$ such that
\begin{itemize}
\item $\overline{R}_{\bar{\mathcal{X}}}(0_{\bar{\mathcal{X}}}) = \bar{\mathcal{X}}$, where $0_{\bar{\mathcal{X}}}$ is the zero vector in $T_{\bar{\mathcal{X}}}\overline{\mathcal{M}}$;
\item $D\overline{R}_{\bar{\mathcal{X}}}(0_{\bar{\mathcal{X}}}) = \operatorname{id}_{T_{\bar{\mathcal{X}}}\overline{\mathcal{M}}}$, the identity map on the tangent space.
\end{itemize}

If a retraction $\overline{R}$ on the total space additionally satisfies
\begin{equation}\label{eq:ret-compat}
\pi\bigl(\overline{R}_{\bar{\mathcal{X}}_a}(\overline{\xi}_{\bar{\mathcal{X}}_a})\bigr) = \pi\bigl(\overline{R}_{\bar{\mathcal{X}}_b}(\overline{\xi}_{\bar{\mathcal{X}}_b})\bigr)
\end{equation}
for any $\bar{\mathcal{X}}_a,\bar{\mathcal{X}}_b$ in the same fibre $\pi^{-1}([\bar{\mathcal{X}}])$ and any horizontal lift $\overline{\xi}$, then
$R_{[\bar{\mathcal{X}}]}(\xi) := \pi\bigl(\overline{R}_{\bar{\mathcal{X}}}(\overline{\xi}_{\bar{\mathcal{X}}})\bigr)$
defines a retraction on the quotient manifold \cite[Proposition 4.1.3]{absil2008optimization}.
For our specific total space $\overline{\mathcal{M}}_{\rr}$, this compatibility condition is indispensable to ensure that the retraction descends to the quotient manifold $\mathcal{M}_{\rr}$.

A natural attempt is to treat $\overline{\mathcal{M}}_{\rr}$ as a product of Stiefel manifolds and apply a standard retraction (e.g., polar or QR decomposition) to each core independently. While this yields a valid retraction on the total space, it is not guaranteed to satisfy the quotient compatibility \eqref{eq:ret-compat}. For instance, the QR-based retraction on the Stiefel manifold is not guaranteed to satisfy the compatibility condition. The polar-based variant, given explicitly by
\[
\overline{R}^1_{\bar{\mathcal{X}}}(\overline{\xi}) = \Bigl(L^{-1}\!\bigl(\operatorname{uf}(L(\mathcal{X}_1+\xi_1))\bigr),\dots,
L^{-1}\!\bigl(\operatorname{uf}(L(\mathcal{X}_{d-1}+\xi_{d-1}))\bigr),\,
\mathcal{X}_d+\xi_d\Bigr),
\]
where $\operatorname{uf}(A)$ extracts the orthogonal factor of the polar decomposition, does satisfy the compatibility condition~\eqref{eq:ret-compat}; a verification of this fact is implicit in the quotient-geometric frameworks for Tucker tensors that employ an analogous retraction~\cite{kasai2016low}. However, this compatibility follows from the particular algebraic properties of the polar decomposition rather than from a general construction principle. Consequently, this product-wise approach does not offer a systematic way to design retractions that are compatible with the quotient structure for TT tensors.

We propose a general strategy to overcome this difficulty. The key observation is that the compatibility condition~\eqref{eq:ret-compat} is automatically fulfilled if the retraction preserves the full tensor, i.e.,
\begin{equation}\label{eq:tensor-preserving}
\Phi\bigl(\overline{R}_{\bar{\mathcal{X}}}(\overline{\xi}_{\bar{\mathcal{X}}})\bigr) = \Phi(\bar{\mathcal{X}} + \overline{\xi}_{\bar{\mathcal{X}}}),
\end{equation}
where $\overline{\xi}_{\bar{\mathcal{X}}}$ is the horizontal lift at $\bar{\mathcal{X}}$. Indeed, by Lemma~\ref{hlift}, for any $\bar{\mathcal{Y}} = \varphi_Q(\bar{\mathcal{X}})$ in the same fibre we have $\overline{\xi}_{\bar{\mathcal{Y}}} = \varphi_Q(\overline{\xi}_{\bar{\mathcal{X}}})$, and the invariance of $\Phi$ under the group action gives $\Phi(\bar{\mathcal{Y}} + \overline{\xi}_{\bar{\mathcal{Y}}}) = \Phi(\bar{\mathcal{X}} + \overline{\xi}_{\bar{\mathcal{X}}})$. Consequently, if $\overline{R}$ satisfies~\eqref{eq:tensor-preserving}, then $\Phi(\overline{R}_{\bar{\mathcal{Y}}}(\overline{\xi}_{\bar{\mathcal{Y}}})) = \Phi(\overline{R}_{\bar{\mathcal{X}}}(\overline{\xi}_{\bar{\mathcal{X}}}))$ follows immediately, which is precisely the compatibility condition~\eqref{eq:ret-compat}.

This observation converts the problem from verifying the quotient compatibility into constructing a retraction that satisfies the formula~\eqref{eq:tensor-preserving}. The latter is a much more tractable task: verifying the original compatibility condition requires checking the equality of two retractions at arbitrarily chosen representatives, whereas constructing a mapping to satisfy~\eqref{eq:tensor-preserving} is a design problem with a clear target. Moreover, the TT format already provides such a procedure---the recursive left-orthogonalization process~\cite{oseledets2011tensor,steinlechner2016riemannian} was designed precisely to produce left-orthogonal cores without altering the full tensor, i.e., it satisfies~\eqref{eq:tensor-preserving} by construction. Adapting this procedure into retractions yields the following two mappings.

\textbf{Recursive polar retraction.} The first variant mimics left-orthogonalization via the polar decomposition:

\[
\overline{R}^2_{\bar{\mathcal{X}}}(\overline{\xi}) = \Bigl(L^{-1}\!\bigl(\operatorname{uf}(L(\mathcal{X}_1+\xi_1))\bigr),\,
L^{-1}\!\bigl(\operatorname{uf}(L((\mathcal{X}_2+\xi_2)\times_1 D_1))\bigr),\dots,\,
(\mathcal{X}_d+\xi_d)\times_1 D_{d-1}\Bigr),
\]
where $D_1 = \operatorname{up}(L(\mathcal{X}_1+\xi_1))$ and for $k=2,\dots,d-1$,
$D_k = \operatorname{up}\!\bigl(L((\mathcal{X}_k+\xi_k)\times_1 D_{k-1}^{\!\top})\bigr)$.
Here $\operatorname{up}(A) = (A^{\!\top}A)^{1/2}$ is the positive definite factor from the polar decomposition.

\textbf{Recursive QR retraction.} The second variant replaces the polar decomposition by the QR decomposition:

\[
\overline{R}^3_{\bar{\mathcal{X}}}(\overline{\xi}) = \Bigl(L^{-1}\!\bigl(\operatorname{qf}(L(\mathcal{X}_1+\xi_1))\bigr),\,
L^{-1}\!\bigl(\operatorname{qf}(L((\mathcal{X}_2+\xi_2)\times_1 F_1^{\!\top}))\bigr),\dots,\,
(\mathcal{X}_d+\xi_d)\times_1 F_{d-1}^{\!\top}\Bigr),
\]
where $F_1$ is the $R$-factor from the QR decomposition of $L(\mathcal{X}_1+\xi_1)$, and $F_k$ is the $R$-factor from the QR decomposition of $L((\mathcal{X}_k+\xi_k)\times_1 F_{k-1}^{\!\top})$ for $k=2,\dots,d-1$.
Here $\operatorname{qf}(A)$ denotes the $Q$-factor of the QR decomposition of $A$.

\begin{proposition}\label{3-overRe}
The mappings $\overline{R}^2$ and $\overline{R}^3$ defined above are retractions on $\overline{\mathcal{M}}_{\rr}$.
\end{proposition}
\begin{proof}
We prove for $\overline{R}^3$; the case of $\overline{R}^2$ is analogous. By definition $\overline{R}^3_{\bar{\mathcal{X}}}(0)=\bar{\mathcal{X}}$. From the construction,
\begin{equation}\label{A1}
L(\mathcal{X}_1 + t\xi_1) = \operatorname{qf}\!\bigl(L(\mathcal{X}_1 + t\xi_1)\bigr) \, F_1(t),
\end{equation}
where $F_1(t)$ is the $R$-factor of the QR decomposition, with $F_1(0)=I_{r_1}$. Since $\operatorname{qf}$ is a retraction on the Stiefel manifold,
\begin{equation}\label{A2}
\operatorname{qf}\!\bigl(L(\mathcal{X}_1 + t\xi_1)\bigr) = L(\mathcal{X}_1 + t\xi_1) + \mathcal{O}(t^2).
\end{equation}
Substituting \eqref{A2} into \eqref{A1}, differentiating with respect to $t$, and evaluating at $t=0$ gives $F'_1(0)=0$, hence $F_1(t) = I_{r_1} + \mathcal{O}(t^2)$.

Using the unfolding relation (\ref{eq:Ltrans}) and $F_1(t)=I+\mathcal{O}(t^2)$,
\[
\operatorname{qf}\!\bigl(L((\mathcal{X}_2 + t\xi_2) \times_1 F_1(t)^{\!\top})\bigr)
= \operatorname{qf}\!\bigl(L(\mathcal{X}_2 + t\xi_2) + \mathcal{O}(t^2)\bigr)
= L(\mathcal{X}_2 + t\xi_2) + \mathcal{O}(t^2).
\]
Proceeding recursively yields, for $k=2,\dots,d-1$,
\[
L^{-1}\!\Bigl(\operatorname{qf}\!\bigl(L((\mathcal{X}_k + t\xi_k) \times_1 F_{k-1}(t)^{\!\top})\bigr)\Bigr) = \mathcal{X}_k + t\xi_k + \mathcal{O}(t^2),
\]
and similarly $(\mathcal{X}_d + t\xi_d) \times_1 F_{d-1}(t)^{\!\top} = \mathcal{X}_d + t\xi_d + \mathcal{O}(t^2)$. Thus $\overline{R}^3_{\bar{\mathcal{X}}}(t\overline{\xi}) = \bar{\mathcal{X}} + t\overline{\xi} + \mathcal{O}(t^2)$, giving $D\overline{R}^3_{\bar{\mathcal{X}}}(0)[\overline{\xi}] = \overline{\xi}$.
\end{proof}

All three retractions share the same leading-order computational cost of $\mathcal{O}(d n r^3)$. Because $\overline{R}^2$ and $\overline{R}^3$ are derived from a tensor-preserving orthogonalization procedure, they automatically satisfy \eqref{eq:tensor-preserving} and thus induce well-defined retractions on the quotient manifold.

\begin{proposition}\label{3-Re}
Define $R^i_{[\bar{\mathcal{X}}]}(\xi) := \pi(\overline{R}^i_{\bar{\mathcal{X}}}(\overline{\xi}_{\bar{\mathcal{X}}}))$ for $i=2,3$. Then $R^2$ and $R^3$ are retractions on $\mathcal{M}_{\rr}$.
\end{proposition}

\begin{proof}
We show the result for $R^3$; the case of $R^2$ is analogous. Let $\bar{\mathcal{X}},\bar{\mathcal{Y}}\in\pi^{-1}([\bar{\mathcal{X}}])$ be two representatives of the same equivalence class, with $\bar{\mathcal{Y}}=\varphi_Q(\bar{\mathcal{X}})$. By Lemma~\ref{hlift}, the horizontal lifts satisfy $\overline{\xi}_{\bar{\mathcal{Y}}} = \varphi_Q(\overline{\xi}_{\bar{\mathcal{X}}})$. From the tensor-preserving property \eqref{eq:tensor-preserving} of $\overline{R}^3$, we have
\begin{align*}
\Phi\bigl(\overline{R}^3_{\bar{\mathcal{Y}}}(\overline{\xi}_{\bar{\mathcal{Y}}})\bigr)
&= \Phi\bigl(\bar{\mathcal{Y}} + \overline{\xi}_{\bar{\mathcal{Y}}}\bigr)
= \Phi\bigl(\varphi_Q(\bar{\mathcal{X}}) + \varphi_Q(\overline{\xi}_{\bar{\mathcal{X}}})\bigr) \\
&= \Phi\bigl(\varphi_Q(\bar{\mathcal{X}} + \overline{\xi}_{\bar{\mathcal{X}}})\bigr)
= \Phi\bigl(\bar{\mathcal{X}} + \overline{\xi}_{\bar{\mathcal{X}}}\bigr)
= \Phi\bigl(\overline{R}^3_{\bar{\mathcal{X}}}(\overline{\xi}_{\bar{\mathcal{X}}})\bigr).
\end{align*}
Hence $\pi\bigl(\overline{R}^3_{\bar{\mathcal{Y}}}(\overline{\xi}_{\bar{\mathcal{Y}}})\bigr) = \pi\bigl(\overline{R}^3_{\bar{\mathcal{X}}}(\overline{\xi}_{\bar{\mathcal{X}}})\bigr)$, and the value of $R^3_{[\bar{\mathcal{X}}]}(\xi_{[\bar{\mathcal{X}}]})$ is independent of the chosen representative. Thus $R^3$ is a well-defined retraction on $\mathcal{M}_{\rr}$.
\end{proof}

While developed for the TT format, this strategy of adapting representation-preserving transformations into retractions is not intrinsically limited to the TT case. It suggests a potential route for designing quotient-compatible retractions in other structured low-rank models, such as Tucker decomposition or matrix factorization, where analogous procedures that preserve the full tensor or matrix are available.

\subsection{Vector Transport}

On a general manifold \(\mathcal{M}\), a vector transport is a smooth map
\[
\mathcal{T}: T\mathcal{M}\oplus T\mathcal{M} \to T\mathcal{M},\quad (\eta_x,\xi_x) \mapsto \mathcal{T}_{\eta_x}\xi_x,
\]
that satisfies certain consistency conditions (see \cite[Section 8.1.4]{absil2008optimization}). Given a retraction \(\overline{R}\) on the total space, a natural vector transport on the quotient manifold \(\mathcal{M}_{\rr}\) can be constructed as follows. Let \(\bar{\mathcal{Y}} = \overline{R}_{\bar{\mathcal{X}}}(\overline{\eta}_{\bar{\mathcal{X}}})\). Then define
\[
\overline{\mathcal{T}_{\eta_{[\bar{\mathcal{X}}]}}\xi_{[\bar{\mathcal{X}}]}}_{\bar{\mathcal{Y}}}
:= \phi_{\bar{\mathcal{Y}}}\!\bigl(\psi_{\bar{\mathcal{Y}}}(\overline{\xi}_{\bar{\mathcal{X}}})\bigr),
\]
where \(\phi\) is the horizontal projection given in Theorem~\ref{hproj} and \(\psi\) is the orthogonal projection given in Theorem~\ref{tproj}. This construction yields a well-defined vector transport \(\mathcal{T}\) on \(\mathcal{M}_{\rr}\).

\section{Optimization Algorithms on the Quotient Manifold}
\label{sec:algorithms}

In this section, we present two first-order Riemannian algorithms for solving the tensor completion problem on the quotient manifold: Riemannian gradient descent (RGD)  and Riemannian conjugate gradient (RCG). The step-size strategies used in these algorithms follow standard approaches from the literature, which we briefly recall below.

\subsection{Riemannian Gradient Descent}\label{sec_lorgd}

We denote the proposed algorithm on the left-orthogonal quotient manifold with Riemannian metric $\va$ as LO-RGD($\va$). The method follows the standard RGD update on the total space:
\[
\bar{\mathcal{X}}_{k+1} = \overline{R}_{\bar{\mathcal{X}}_k}(\alpha_k \overline{\xi}_k),
\]
where $\overline{\xi}_k = -\operatorname{grad}\overline{h}(\bar{\mathcal{X}}_k)$ is the negative Riemannian gradient, $\alpha_k>0$ is the step size, and $\overline{R}$ is one of the retractions defined in Section~\ref{sec:retraction}. 

For the step size $\alpha_k$, we consider three standard strategies: Armijo-quad obtains an initial step by minimizing a quadratic model of $\overline{h}$ along the search direction, followed by Armijo backtracking~\cite{cai2022tensor,kasai2016low,kressner2014low}; Armijo-NW guesses an initial step from the function decrease in the previous iteration, followed by the same Armijo backtracking~\cite{nocedal2006numerical,boumal2014manopt}; BB uses the Riemannian Barzilai-Borwein step size~\cite{iannazzo2018riemannian}, combined with a nonmonotone line search from Manopt~\cite{boumal2014manopt}.
The complete procedure is listed in Algorithm~\ref{alg_rgd}.

\begin{algorithm}[htbp]
\caption{Riemannian Gradient Descent (LO-RGD($\va$))}
\label{alg_rgd}
\begin{algorithmic}[1]
\State \textbf{Input:} Initial iterate $\bar{\mathcal{X}}_0\in\overline{\mathcal{M}}_{\rr}$; maximum iterations $k_{\max}$; a retraction $\overline{R}$ chosen from $\overline{R}^1,\overline{R}^2,\overline{R}^3$; a metric parameter $\va$; a step size strategy from $\{\text{Armijo-quad},\text{Armijo-NW},\text{BB}\}$.
\State \textbf{Output:} Sequence of iterates $\{\bar{\mathcal{X}}_k\}$.
\For{$k=0,1,\dots,k_{\max}$}
    \State $\overline{\xi}_k \gets -\operatorname{grad}\overline{h}(\bar{\mathcal{X}}_k)$
    \State Compute the step size $\alpha_k$ using the chosen strategy
    \State $\bar{\mathcal{X}}_{k+1} \gets \overline{R}_{\bar{\mathcal{X}}_k}(\alpha_k\overline{\xi}_k)$
\EndFor
\end{algorithmic}
\end{algorithm}

For the full metric, computing the Riemannian gradient requires $\mathcal{O}(d|\Omega| r^2)$ operations for the sparse Euclidean gradient evaluation using~\cite[Algorithm~2]{steinlechner2016riemannian}, $\mathcal{O}(d n r^3)$ for forming $R^{k+1}$ and the subsequent multiplications, and $\mathcal{O}(d r^3)$ for the inversions. The tangent projection adds another $\mathcal{O}(d n r^3 + d r^3)$. Thus the total per-iteration cost for the Riemannian gradient is $\mathcal{O}(d|\Omega| r^2+d n r^3)$. The costs of the retraction depend on the specific choice, as analyzed in Section~\ref{sec:retraction}; the step size strategies contribute a lower-order cost that depends on the number of backtracking trials.

\subsection{Riemannian Conjugate Gradient}\label{sec_lorcg}

The proposed algorithm on the left-orthogonal quotient manifold with Riemannian metric $\va$, denoted LO-RCG($\va$), applies the RCG method. It uses a search direction that combines the current negative gradient with a transported previous direction~\cite{sato2015new}:
\[
\overline{\xi}_k = -\operatorname{grad}\overline{h}(\bar{\mathcal{X}}_k) + \beta_k\,\phi_{\bar{\mathcal{X}}_k}\!\bigl(\psi_{\bar{\mathcal{X}}_k}(\overline{\xi}_{k-1})\bigr),
\]
where $\beta_k$ is chosen according to the Hestenes-Stiefel formula
\[
\beta_k = \max\!\Bigl\{0,\;
\frac{\overline{g}^{\va}_{\bar{\mathcal{X}}_k}\!\bigl(\operatorname{grad}\overline{h}(\bar{\mathcal{X}}_k)-\phi_{\bar{\mathcal{X}}_k}(\psi_{\bar{\mathcal{X}}_k}(\operatorname{grad}\overline{h}(\bar{\mathcal{X}}_{k-1}))),\;
\operatorname{grad}\overline{h}(\bar{\mathcal{X}}_k)\bigr)}
{\overline{g}^{\va}_{\bar{\mathcal{X}}_k}\!\bigl(\operatorname{grad}\overline{h}(\bar{\mathcal{X}}_k)-\phi_{\bar{\mathcal{X}}_k}(\psi_{\bar{\mathcal{X}}_k}(\operatorname{grad}\overline{h}(\bar{\mathcal{X}}_{k-1}))),\;
\phi_{\bar{\mathcal{X}}_k}(\psi_{\bar{\mathcal{X}}_k}(\overline{\xi}_{k-1}))\bigr)}
\Bigr\}.
\]
The update is again $\bar{\mathcal{X}}_{k+1} = \overline{R}_{\bar{\mathcal{X}}_k}(\alpha_k\overline{\xi}_k)$. For the step size, we employ the Armijo-quad and Armijo-NW strategies described in Section~\ref{sec_lorgd}. The complete algorithm is outlined in Algorithm~\ref{alg_rcg}.

\begin{algorithm}[htbp]
\caption{Riemannian Conjugate Gradient (LO-RCG($\va$))}
\label{alg_rcg}
\begin{algorithmic}[1]
\State \textbf{Input:} Initial iterate $\bar{\mathcal{X}}_0\in\overline{\mathcal{M}}_{\rr}$; maximum iterations $k_{\max}$; a retraction $\overline{R}$ chosen from $\overline{R}^1,\overline{R}^2,\overline{R}^3$; a metric parameter $\va$; a step size strategy from $\{\text{Armijo-quad},\text{Armijo-NW}\}$.
\State \textbf{Output:} Sequence of iterates $\{\bar{\mathcal{X}}_k\}$.
\State $\overline{\xi}_0 \gets -\operatorname{grad}\overline{h}(\bar{\mathcal{X}}_0)$
\For{$k=0,1,\dots,k_{\max}$}
    \State Compute the step size $\alpha_k$ using the chosen strategy
    \State $\bar{\mathcal{X}}_{k+1} \gets \overline{R}_{\bar{\mathcal{X}}_k}(\alpha_k\overline{\xi}_k)$
    \State Compute $\beta_{k+1}$ using the Hestenes-Stiefel formula and set
    \State $\overline{\xi}_{k+1} \gets -\operatorname{grad}\overline{h}(\bar{\mathcal{X}}_{k+1}) + \beta_{k+1}\,\phi_{\bar{\mathcal{X}}_{k+1}}\!\bigl(\psi_{\bar{\mathcal{X}}_{k+1}}(\overline{\xi}_k)\bigr)$
\EndFor
\end{algorithmic}
\end{algorithm}

For the full metric, the main computational cost of LO-RCG($\va$) lies in computing the Riemannian gradient (same as for LO-RGD($\va$)) and the horizontal projection. The horizontal projection requires constructing the coefficient matrix, which costs $\mathcal{O}(d n r^5)$ operations, and solving the linear system~\eqref{eq:hmetrix} directly costs $\mathcal{O}(d^3 r^6)$ operations; exploiting the block-tridiagonal structure of the coefficient matrix can further reduce the solution cost. The costs of the retraction and the step size strategy are as discussed in Section~\ref{sec:retraction} and Section~\ref{sec_lorgd}, respectively.

\begin{remark}
For the convergence analysis, we lift the original problem to the quotient manifold $\mathcal{M}_{\rr}$. Thanks to our constructions, the relevant geometric objects on $\mathcal{M}_{\rr}$ (e.g., the Riemannian metric, the retraction, etc.) are well-defined and satisfy the standard assumptions of Riemannian optimization theory (first-order retraction, smoothness of the metric, and suitable step size strategies). Consequently, $\mathcal{M}_{\rr}$, as a Riemannian manifold, can directly invoke the standard convergence results \cite{absil2008optimization,iannazzo2018riemannian,ring2012optimization,sato2015new}.
\end{remark}

\section{Numerical Experiments}
\label{sec:experiments}
To quantify reconstruction accuracy and efficiency, we use the relative error \(\mathrm{RE}= \|\mathcal{P}_\Omega(\mathcal{X}) - \mathcal{P}_\Omega(\Gamma)\|_F / \|\mathcal{P}_\Omega(\Gamma)\|_F\) and the sampling ratio \(\mathrm{SR}= |\Omega|/(n_1 \times \cdots \times n_d)\), where \(\Gamma\) is the target tensor with dimensions \(n_1,\dots,n_d\).

All algorithms are implemented within the Manopt toolbox~\cite{boumal2014manopt}, which calls the MATLAB Tensor Toolbox~\cite{kolda2017tensor} for core tensor operations. All experiments are conducted in MATLAB 2025a on an Apple Mac computer equipped with an Apple M4 Pro chip and 24 GB of RAM.

The experimental evaluation is organized as follows. In Section~\ref{sec:components}, we analyze the influence of the Riemannian metric and retraction map within the proposed algorithmic framework. Section~\ref{sec_condition_number} compares our algorithms with existing TT-based geometric methods under challenging conditions of low oversampling and high condition numbers.

\subsection{Algorithmic Framework}
\label{sec:components}
In this section, we analyze the influence of the Riemannian metric and the retraction map on the proposed algorithms. Throughout this section, we use LO-RGD($\va$) and LO-RCG($\va$) to denote the algorithms with metric $\va$, while the specific choice of retraction is indicated in each subsection.

\subsection*{Performance under Different Riemannian Metrics}
\addcontentsline{toc}{subsection}{Performance under Different Riemannian Metrics}
\label{sec_metric}

The Riemannian metric plays a fundamental role in determining the performance of optimization algorithms on manifolds. In this part, we investigate how different choices of the metric affect the convergence behavior of the proposed methods, while keeping the step size and retraction fixed. Specifically, we employ the Armijo-NW strategy described in Section~\ref{sec_lorgd} and the retraction $\overline{R}^1$ throughout. For reference, we also include the quotient-geometry algorithms RGD(Q) and RCG(Q) from~\cite{cai2022tensor}, whose quotient manifold construction is reviewed in Section~2, and which employ the full metric by construction.

For third-order tensors, we construct a family of six admissible Riemannian metrics on the left-orthogonal quotient manifold, denoted by $\va_1,\dots,\va_6$ (see Section~\ref{sec:metric_family} for their definitions). Here $\va_1$ corresponds to the Euclidean metric and $\va_6$ to the full metric. We compare the performance of LO-RGD($\va_i$) and LO-RCG($\va_i$) under each of these six metrics, alongside RGD(Q) and RCG(Q). The sampling ratio is fixed at $\mathrm{SR}=0.4$ for all tests. The maximum number of iterations is set to $100$, and the optimization is terminated early when the norm of the Riemannian gradient falls below $10^{-6}$.

To investigate how the sensitivity to the metric depends on the core tensor dimensions, we consider two configurations capturing different patterns of dimensional imbalance: one with a small third core (tensor size $[100,100,5]$, TT-rank $[1,10,3,1]$), and one with both the second and third cores small (tensor size $[500,10,10]$, TT-rank $[1,3,3,1]$).

For each setting, the target tensor is generated using the following generic procedure (denoted as the ``generic generation'' method and also used in later sections). First, TT core tensors of appropriate sizes are sampled randomly. For numerical stability, the left-unfolding matrices of the first $d-1$ cores are then orthogonalized via QR factorization. The target tensor is finally formed by taking the TT product of the cores. Initial iterates are constructed in the same manner.
\label{sec:generic_generation}

The convergence behavior under the six candidate metrics, together with RGD(Q) and RCG(Q), is compared in Figures~\ref{fig:metrics_small3} and~\ref{fig:metrics_small23}. Figure~\ref{fig:metrics_small3} corresponds to the small-third-core configuration, while Figure~\ref{fig:metrics_small23} corresponds to the case where both the second and third cores are small. Each figure contains four subfigures, displaying the relative error $\mathrm{RE}$ against the number of iterations (left column) and against the computational time in seconds (right column), for RGD (top row) and RCG (bottom row). The six metrics $\va_1,\dots,\va_6$ are distinguished by different line styles and colors, while RGD(Q) and RCG(Q) are shown as thicker black curves with distinct markers for clear identification as baselines.

\begin{figure}[htbp]
\centering
\subfigure[RGD: RE vs iteration]{\includegraphics[width=0.45\textwidth]{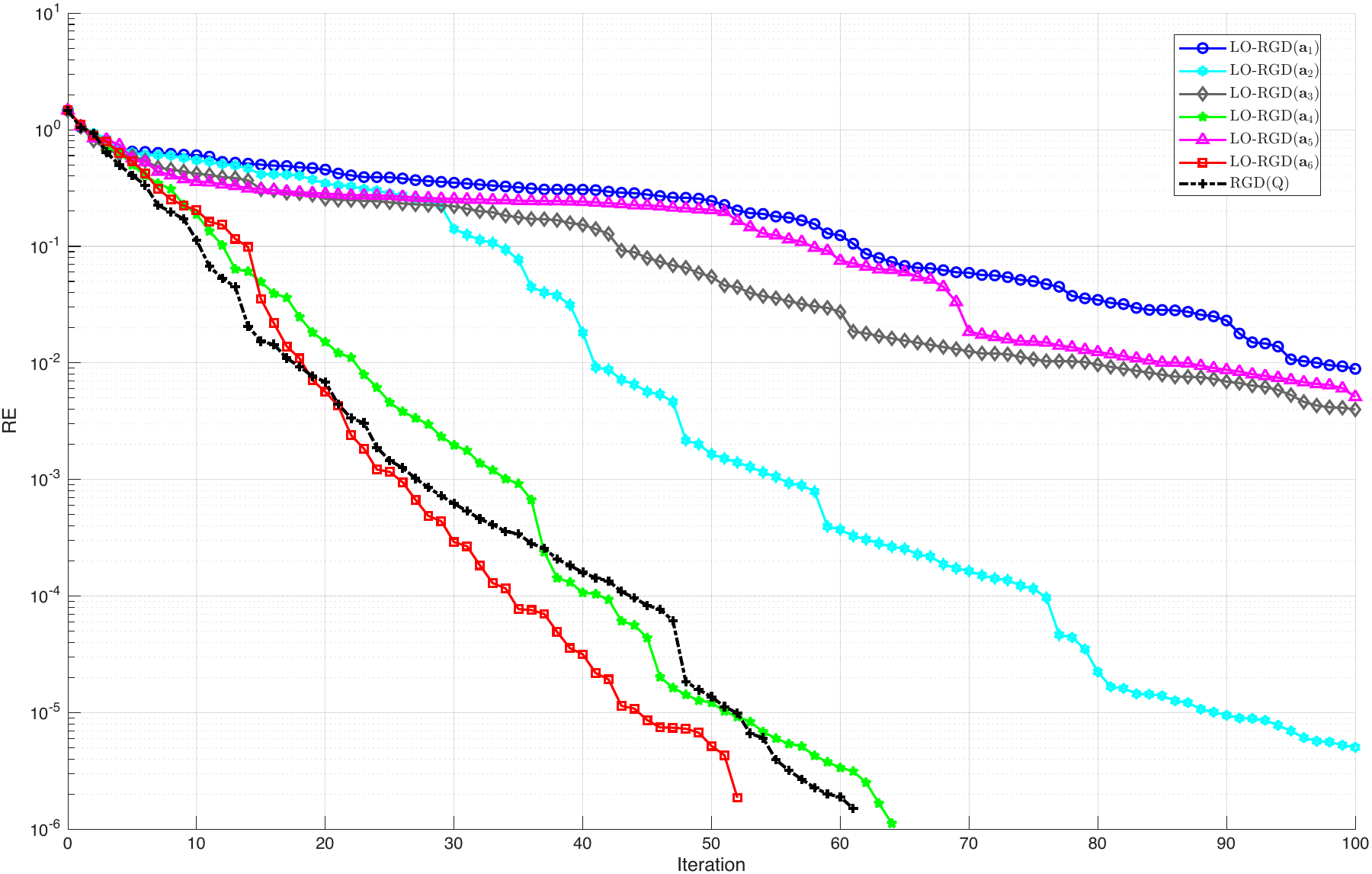}}
\hfill
\subfigure[RGD: RE vs time]{\includegraphics[width=0.45\textwidth]{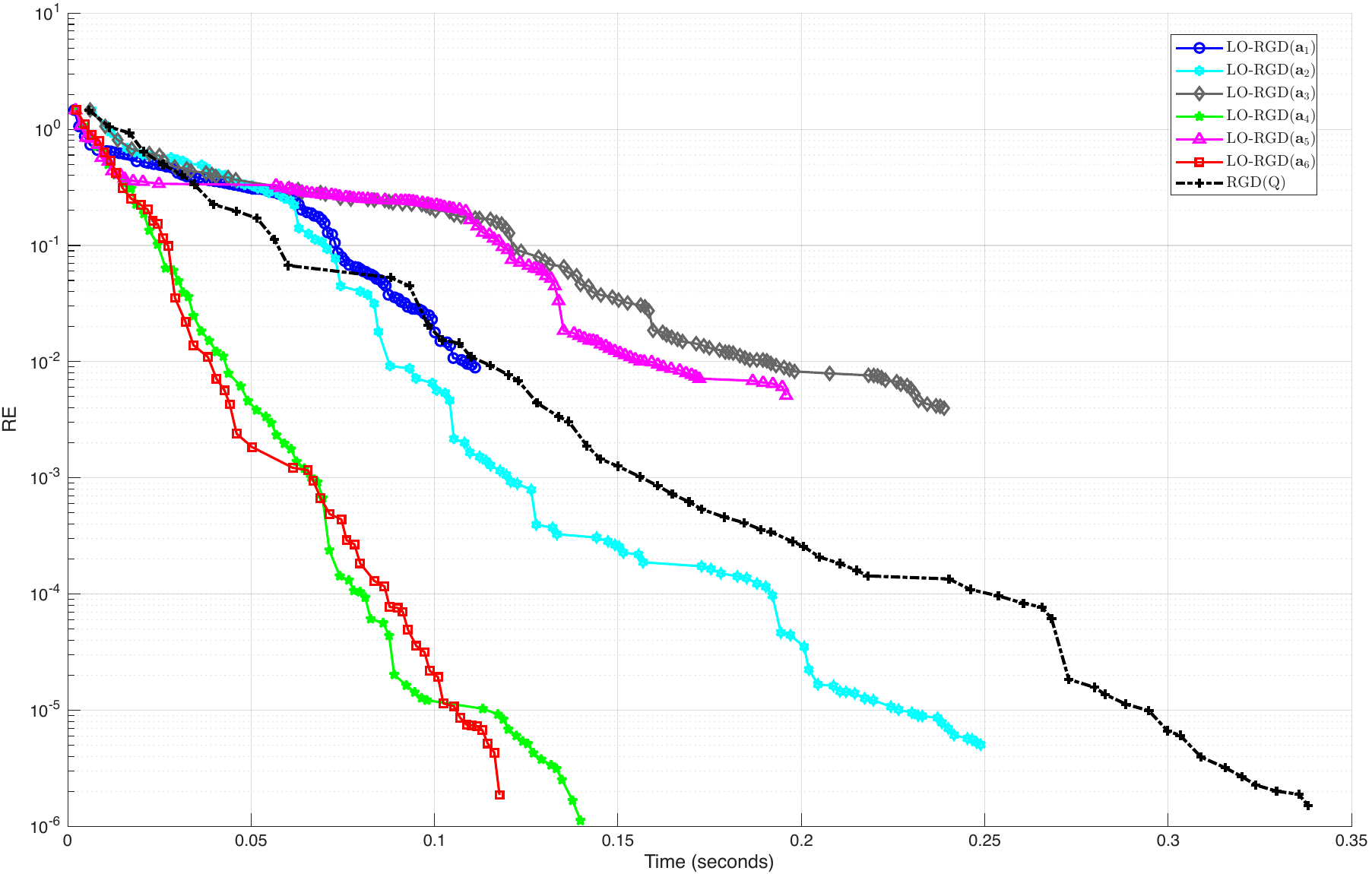}}

\vspace{0.3cm}

\subfigure[RCG: RE vs iteration]{\includegraphics[width=0.45\textwidth]{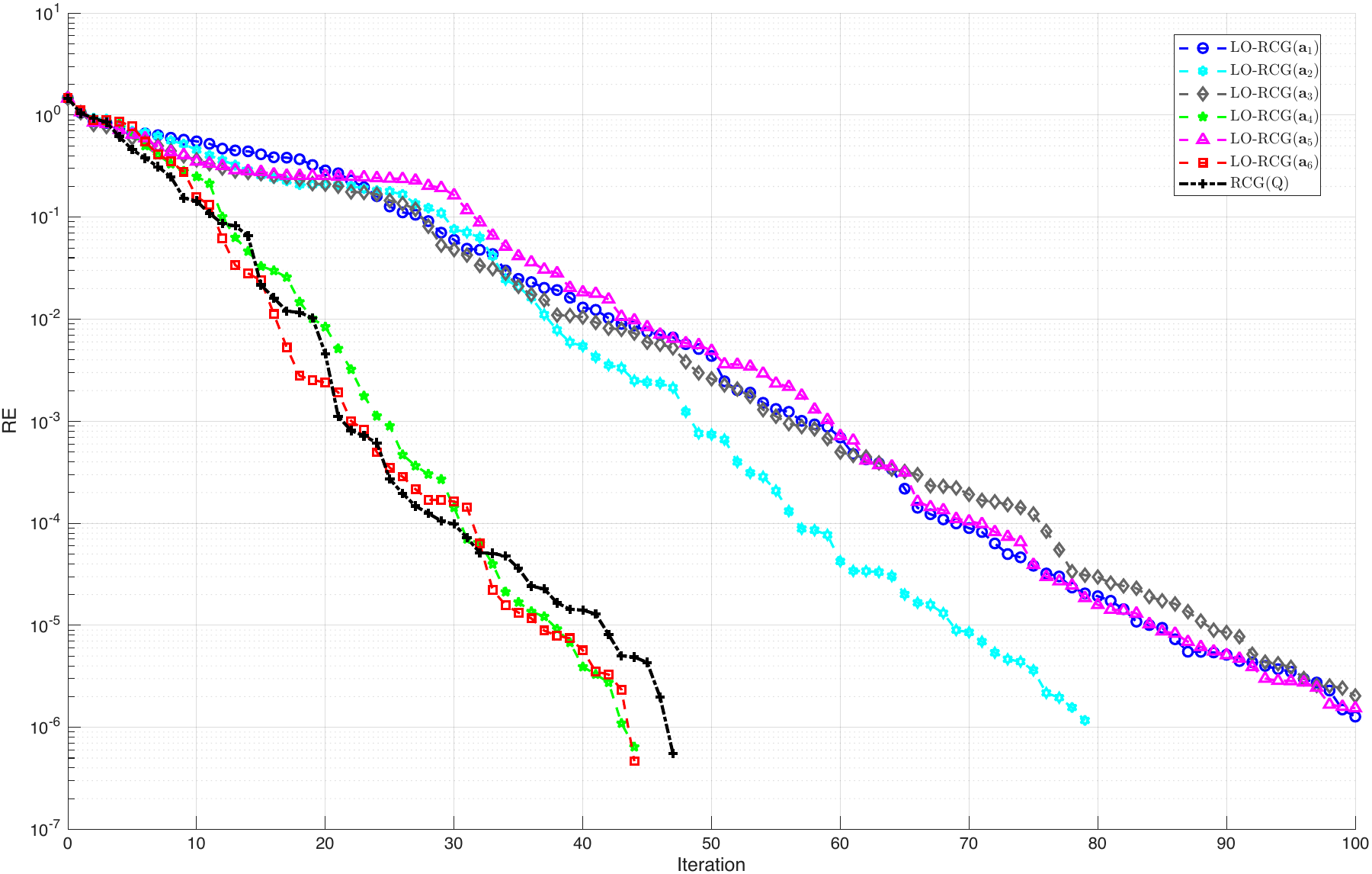}}
\hfill
\subfigure[RCG: RE vs time]{\includegraphics[width=0.45\textwidth]{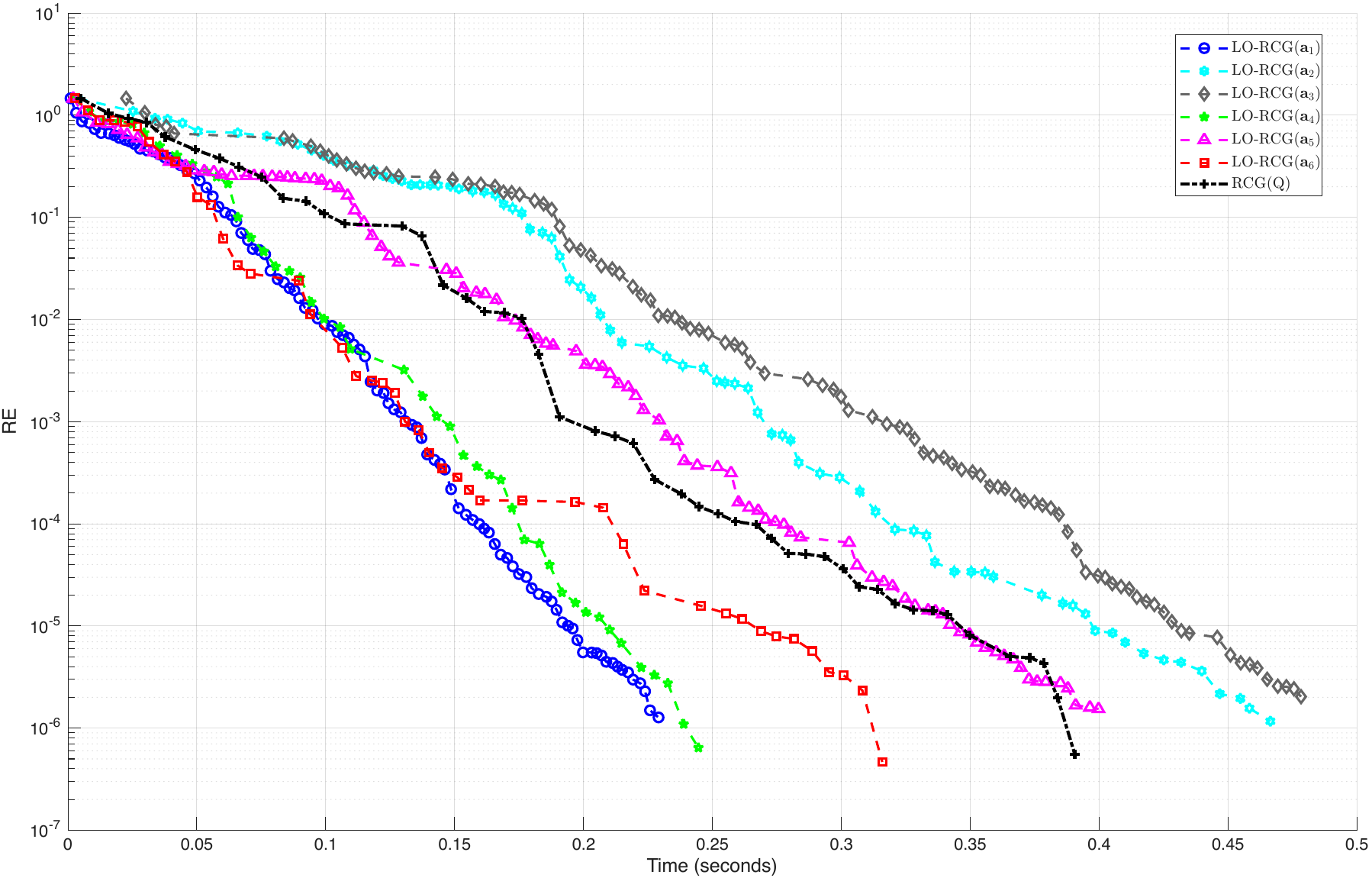}}
\caption{Convergence behavior under six Riemannian metrics for the small-third-core configuration (size $[100,100,5]$, TT-rank $[1,10,3,1]$). The full metric $\va_6$ decays fastest per iteration but its per-step cost is highest; simpler metrics become more efficient in time.}
\label{fig:metrics_small3}
\end{figure}

\begin{figure}[htbp]
\centering
\subfigure[RGD: RE vs iteration]{\includegraphics[width=0.45\textwidth]{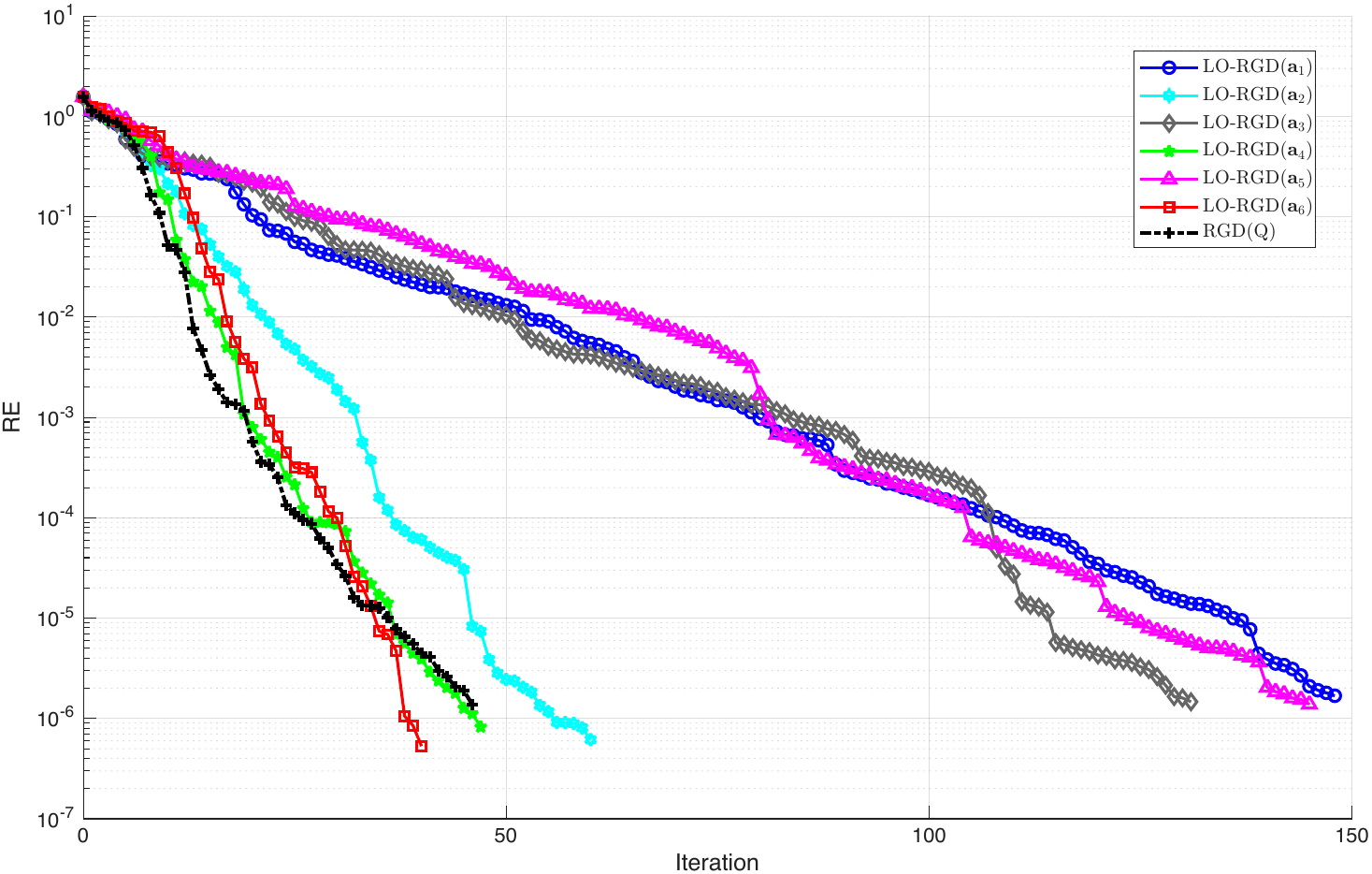}}
\hfill
\subfigure[RGD: RE vs time]{\includegraphics[width=0.45\textwidth]{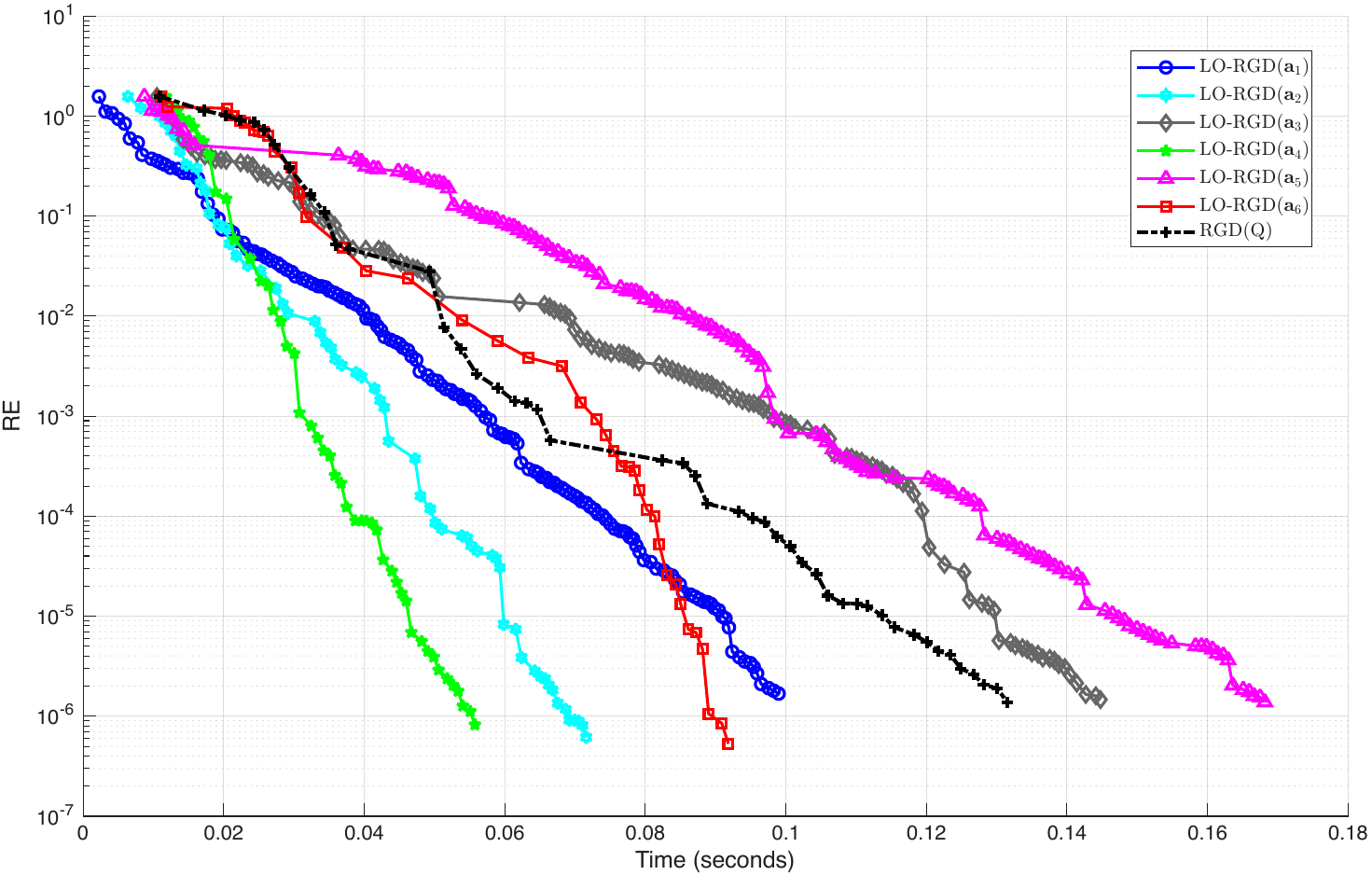}}

\vspace{0.3cm}

\subfigure[RCG: RE vs iteration]{\includegraphics[width=0.45\textwidth]{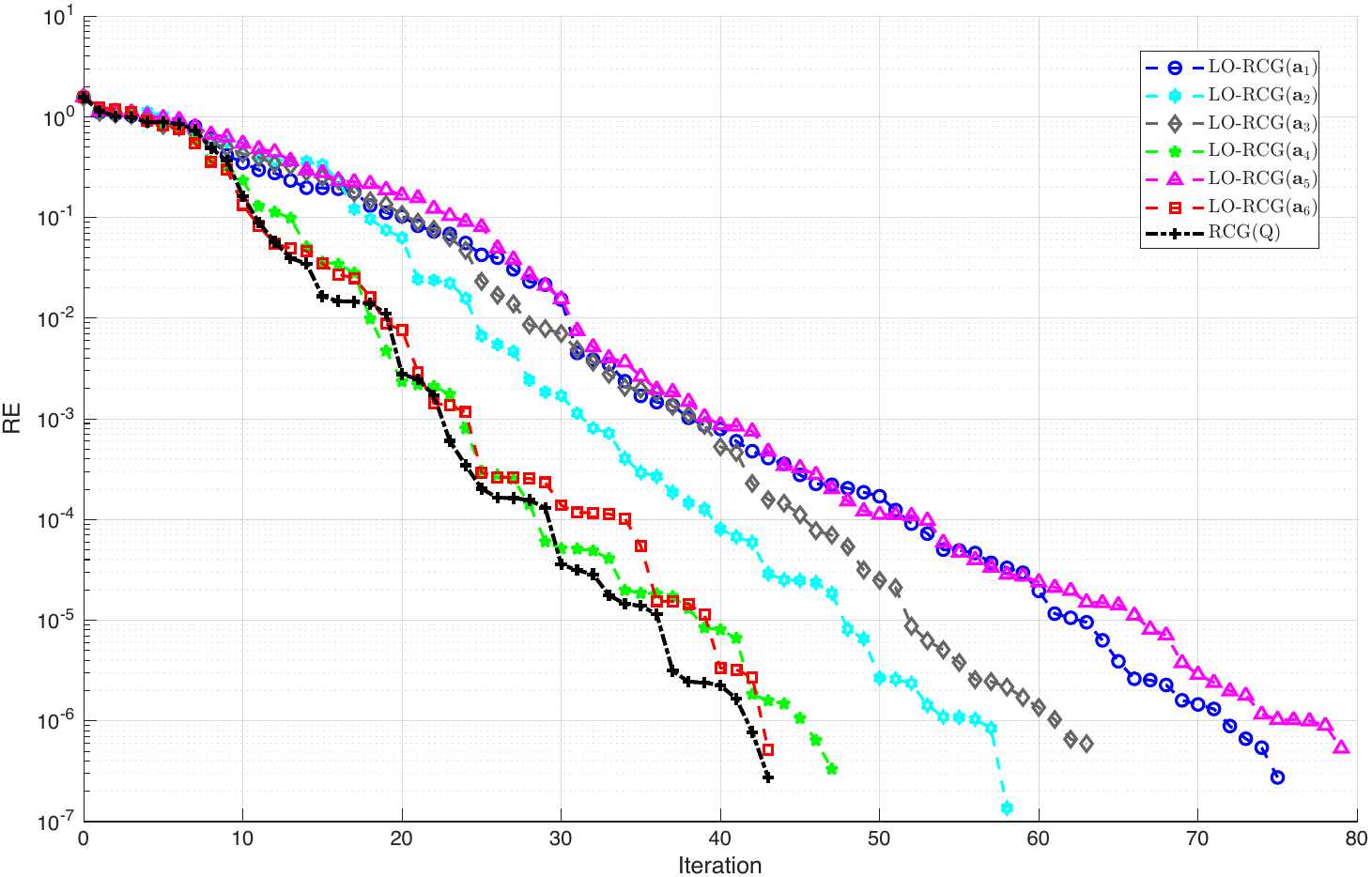}}
\hfill
\subfigure[RCG: RE vs time]{\includegraphics[width=0.45\textwidth]{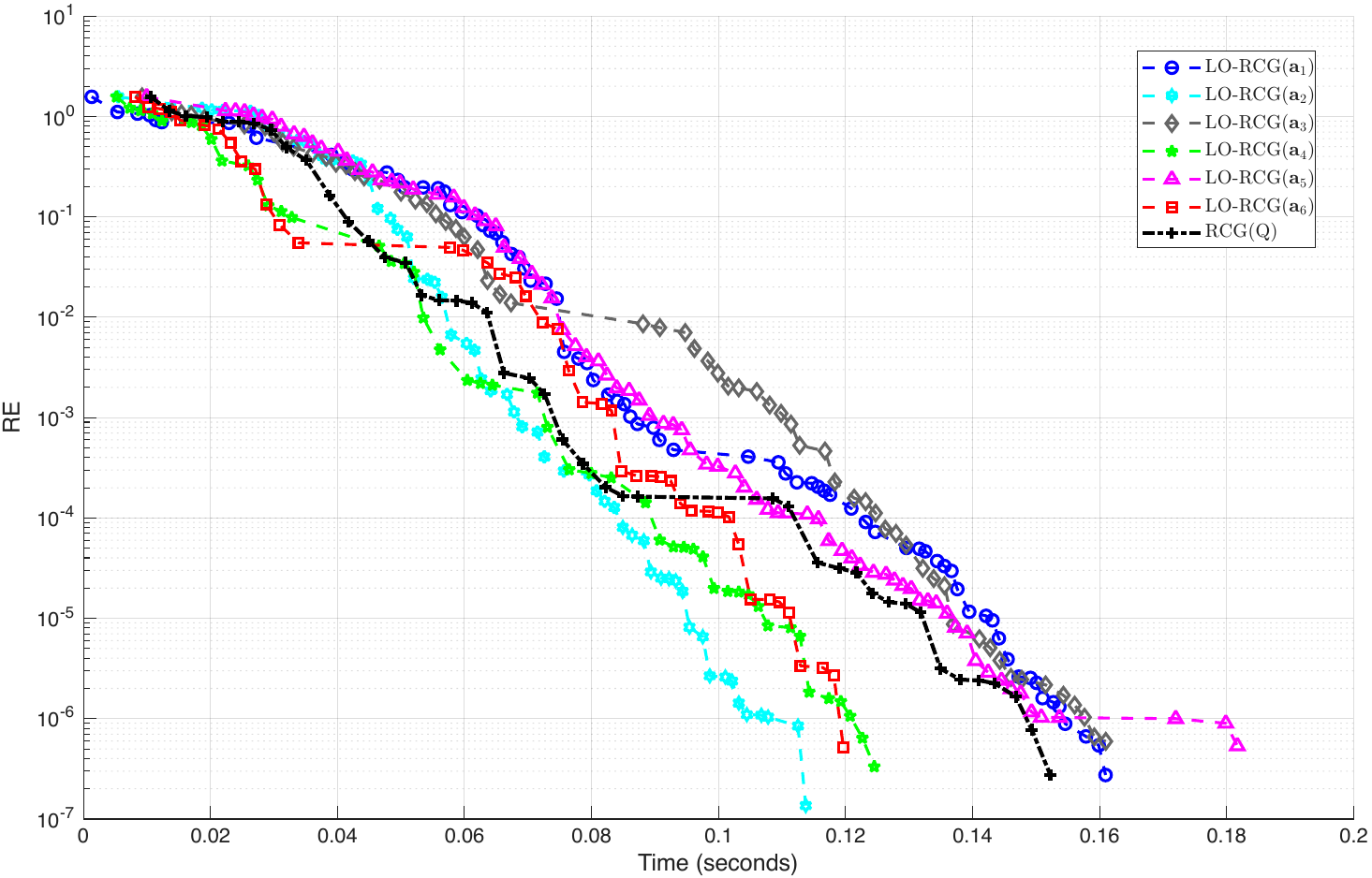}}
\caption{Convergence behavior under six Riemannian metrics for the configuration with small second and third cores (size $[500,10,10]$, TT-rank $[1,3,3,1]$). All metrics perform similarly; the advantage of $\va_6$ diminishes when the cores involved are small.}
\label{fig:metrics_small23}
\end{figure}

From Figures~\ref{fig:metrics_small3} and~\ref{fig:metrics_small23}, we draw the following observations. 
All six proposed Riemannian metrics support convergence on the left-orthogonal quotient manifold. 
Among them, the full metric $\va_6$, which couples all TT cores, consistently achieves the fastest per-iteration decay of the relative error. 
Its RE-versus-iteration curves closely track those of the existing quotient methods RGD(Q) and RCG(Q) from~\cite{cai2022tensor}, confirming that the geometric reduction from $\mathrm{GL}(r_k)$ to $\mathrm{O}(r_k)$ does not introduce computational overhead. 
However, the per-iteration cost of $\va_6$ is noticeably higher than that of the other metrics, and its iteration-wise advantage does not automatically translate into time-wise efficiency. 
When the third core is small (Figure~\ref{fig:metrics_small3}), the time advantage of $\va_6$ nearly vanishes, and simpler metrics such as $\va_4$ become equally or more efficient in wall-clock time. 
When both the second and third cores are small (Figure~\ref{fig:metrics_small23}), the iteration advantage of $\va_6$ also diminishes significantly because the additional geometric information it captures becomes marginal, making metrics with lower computational overhead, such as $\va_2$ and $\va_4$, more efficient choices. 
Hence, the full metric should not be chosen indiscriminately: it is preferable when every TT core is large and iteration count is paramount, whereas other metrics from the proposed family offer a better trade-off when certain cores are small or computational resources are limited.

This family of Riemannian metrics offers the flexibility to choose the most appropriate metric according to the specific tensor structure.

\subsection*{Impact of Retraction Choices}

\addcontentsline{toc}{subsection}{Impact of Retraction Choices}
\label{sec_retraction}
In this part, we compare the effects of different retractions on the convergence performance of the algorithms.
Based on the metric comparison, we fix the Riemannian metric to $\va_{d!}$ for the remainder of this section and all subsequent experiments, as it exhibited robust performance across different tensor configurations. The step size strategy is fixed to Armijo with quadratic initial step (Armijo-quad, see Section~\ref{sec_lorgd}) for all variants in this part. We compare the three retractions by running both LO-RGD and LO-RCG with each of $\overline{R}^1$, $\overline{R}^2$, $\overline{R}^3$, denoting the variants by appending \texttt{-r1}, \texttt{-r2}, \texttt{-r3} to the algorithm names.

The target tensor and initial iterates are generated using the generic procedure described in Section~\ref{sec:generic_generation}. We consider a fourth-order tensor of size $[10,10,10,10]$ with TT-rank $[1,4,4,4,1]$ and a sampling ratio of $0.4$. The maximum number of iterations is set to $100$, and the optimization is terminated early when the Riemannian gradient norm falls below $10^{-6}$. Table~\ref{tab:retraction} reports the average time per iteration (in seconds) and the final relative error achieved by each algorithm.

\begin{table}[htbp]
    \centering
    \caption{Average time per iteration and final relative error for each retraction variant.}
    \label{tab:retraction}
    \begin{tabular}{lcc}
        \hline
        Algorithm & Average time per iteration (s) & Final RE \\
        \hline
        \texttt{LO-RGD($\va_{d!}$)-r1} & $2.600 \times 10^{-3}$ & $8.856 \times 10^{-7}$ \\
        \texttt{LO-RGD($\va_{d!}$)-r2} & $3.228 \times 10^{-3}$ & $9.789 \times 10^{-7}$ \\
        \texttt{LO-RGD($\va_{d!}$)-r3} & $1.606 \times 10^{-3}$ & $9.115 \times 10^{-7}$ \\
        \texttt{LO-RCG($\va_{d!}$)-r1} & $3.407 \times 10^{-3}$ & $6.927 \times 10^{-7}$ \\
        \texttt{LO-RCG($\va_{d!}$)-r2} & $5.045 \times 10^{-3}$ & $4.550 \times 10^{-7}$ \\
        \texttt{LO-RCG($\va_{d!}$)-r3} & $3.962 \times 10^{-3}$ & $7.743 \times 10^{-7}$ \\
        \hline
    \end{tabular}
\end{table}

From the experimental results, we observe that all three retraction maps $\overline{R}^1$, $\overline{R}^2$, and $\overline{R}^3$ successfully lead to convergence, exhibiting similar convergence behavior and computational costs.  These results indicate that the choice of retraction is not a critical factor for the problems considered; all three maps ensure effective convergence with comparable efficiency. Therefore, any of them can be selected in practice. In the following experiments, we uniformly adopt $\overline{R}^1$ as the default retraction.

\subsection{Comparison under Low Oversampling and High Condition Numbers}
\label{sec_condition_number}

In this section, we compare the proposed algorithms with the embedded-geometry methods RGD(E) \cite{cai2022provable,wang2019tensor,steinlechner2016riemannian} and RCG(E) \cite{steinlechner2016riemannian}, as well as the quotient-geometry algorithms RGD(Q) and RCG(Q) from \cite{cai2022tensor}, under varying oversampling ratios and condition numbers. All of the aforementioned algorithms are based on the TT decomposition.

The oversampling (OS) ratio is defined as \( |\Omega| / \dim \mathcal{N}_{\mathbf{r}} \), where \(\dim \mathcal{N}_{\mathbf{r}}\) is the manifold dimension given in Section~2.3. The condition number of a tensor $\mathcal{X}$ is $\kappa(\mathcal{X}) = \sigma_{\max}(\mathcal{X})/\sigma_{\min}(\mathcal{X})$, where $\sigma_{\max}(\mathcal{X})$ and $\sigma_{\min}(\mathcal{X})$ denote the largest and smallest nonzero singular values across all the unfoldings $\mathcal{X}_{\langle 1 \rangle}, \dots, \mathcal{X}_{\langle d-1 \rangle}$ \cite{cai2022provable}; a large condition number indicates an ill-conditioned tensor that poses greater challenges for recovery. These two settings probe different aspects of algorithmic robustness: low oversampling tests the ability to recover from incomplete observations, while high condition numbers assess stability against ill-conditioning.

We set the tensor size to $[100,100,100]$, the TT-rank to $[1,5,5,1]$, and consider $\mathrm{OS} \in \{4,8\}$ and $\kappa \in \{25,50,100\}$. Target tensors with prescribed TT-rank and condition number are generated as in~\cite{cai2022tensor}; we also include target tensors generated by the generic procedure described in Section~\ref{sec:generic_generation}.

For all experiments, we use the step size strategies described in Section~\ref{sec_lorgd} (RGD employs BB, while RCG uses Armijo-NW). The maximum number of iterations is set to $250$, and the optimization terminates when the Riemannian gradient norm falls below $10^{-8}$. The initial iterate for each algorithm is generated using the generic procedure described in Section~\ref{sec:generic_generation}. An algorithm is considered successful if $\|\mathcal{X} - \Gamma\|_F / \|\Gamma\|_F$ of the final iterate is less than $10^{-3}$. Each algorithm is run $100$ times independently, and the success rate is computed. The results are reported in Table~\ref{tab:recovery_rates}, where the row labeled ``Generic'' corresponds to target tensors generated by the generic procedure.

\begin{table}[htbp]
\centering
\caption{Recovery rates under different oversampling ratios and condition numbers.}
\label{tab:recovery_rates}
\setlength{\tabcolsep}{3pt}
\begin{tabular}{cccccccc}
\toprule
\multicolumn{2}{c}{} & \multicolumn{6}{c}{Algorithm} \\
\cmidrule(lr){3-8}
OS & $\kappa$ & RGD(E) & RGD(Q) & LO-RGD($\va_{d!}$) & RCG(E) & RCG(Q) & LO-RCG($\va_{d!}$) \\
\midrule
\multirow{4}{*}{4} & Generic & 1 & 1 & 1 & 0.98 & 1 & 1 \\
 & 25 & 0.01 & 0.85 & 0.86 & 0.01 & 0.74 & 0.81 \\
 & 50 & 0 & 0.83 & 0.81 & 0 & 0.64 & 0.68 \\
 & 100 & 0 & 0.75 & 0.79 & 0 & 0.48 & 0.54 \\
\midrule
\multirow{4}{*}{8} & Generic & 1 & 1 & 1 & 1 & 1 & 1 \\
 & 25 & 0.62 & 1 & 1 & 0.70 & 1 & 1 \\
 & 50 & 0.51 & 0.99 & 1 & 0.65 & 0.99 & 1 \\
 & 100 & 0.19 & 0.99 & 0.98 & 0.33 & 0.98 & 1 \\
\bottomrule
\end{tabular}
\end{table}

From the experimental results, we observe that LO-RGD($\va_{d!}$) and LO-RCG($\va_{d!}$) perform comparably to RGD(Q) and RCG(Q) in most test cases, while the embedded-geometry algorithms perform significantly worse. Under the most challenging setting (OS=4, $\kappa=100$), the embedded-geometry algorithms fail completely (recovery rate $0$), whereas all quotient-geometry algorithms maintain recovery rates above $0.48$, with the proposed algorithms exhibiting a slight advantage over RGD(Q) and RCG(Q) (LO-RGD: $0.79$ vs.\ $0.75$; LO-RCG: $0.54$ vs.\ $0.48$).

\section{Conclusion}
\label{sec:conclusion}

We have developed a Riemannian optimization framework for low-rank tensor completion based on the left-orthogonal TT format. By identifying left-orthogonal TT-cores under the action of the orthogonal group, we constructed a quotient manifold and introduced a family of $d!$ admissible Riemannian metrics, ranging from the Euclidean metric to the full metric used in previous works. This family provides a flexible trade-off between computational cost and geometric information, made possible by combining the left-orthogonal structure with the orthogonal group as the quotient group.

We further proposed a new approach to constructing retractions compatible with the quotient structure, by adapting tensor-preserving transformations of the TT format into retractions, realized via two novel retractions based on recursive polar and QR decompositions. Compared with~\cite{cai2022tensor}, our framework reduces the dimension of the linear system in the horizontal projection from $\sum_{k=1}^{d-1} r_k^2$ to $\sum_{k=1}^{d-1} r_k(r_k-1)/2$. Numerical experiments demonstrate that the proposed algorithms achieve reconstruction accuracy comparable to state-of-the-art TT-based manifold optimization methods. Future work includes extending the proposed metric family and retraction strategy to other tensor formats, and investigating the application of Riemannian stochastic algorithms within this geometric framework.


\bibliographystyle{siamplain}
\bibliography{reference}
\end{document}